\documentclass[11pt]{amsart}
\usepackage{amsmath,amssymb,amsthm}
\usepackage{tikz,xstring,ifthen}
\usepackage[numbers,comma,square,sort&compress]{natbib}
\usepackage{mathabx}   
\usepackage{stackengine}
\usepackage{caption}
\usepackage{subcaption} 
\RequirePackage{color}
\RequirePackage[colorlinks,urlcolor=my-blue,linkcolor=my-red,citecolor=my-green]{hyperref}

\numberwithin{figure}{section}

\usepackage{enumerate}

\usepackage{scalerel}    
\usepackage{stmaryrd}   
\usepackage{mathabx}   
\usepackage{nicefrac}
\definecolor{my-blue}{rgb}{0.0,0.0,0.6}
\definecolor{my-red}{rgb}{0.5,0.0,0.0}
\definecolor{my-green}{rgb}{0.0,0.5,0.0}
\definecolor{nicos-red}{rgb}{0.75,0.0,0.0}

\newtheorem{theorem}{\sc Theorem}[section]
\newtheorem{lemma}[theorem]{\sc Lemma}
\newtheorem{proposition}[theorem]{\sc Proposition}
\newtheorem{corollary}[theorem]{\sc Corollary}

\newtheorem{definition}[theorem]{\sc Definition}
\numberwithin{equation}{section}
\theoremstyle{remark}

\newcommand{\be}{\begin{equation}}
\newcommand{\ee}{\end{equation}}
\newcommand{\beq}{\begin{equation}}
\newcommand{\eeq}{\end{equation}}

\providecommand{\abs}[1]{\vert#1\vert}

\def\cC{\mathcal{C}}  \def\cD{\mathcal{D}}
\def\cF{\mathcal{F}}
\def\cG{\mathcal{G}}

\def\cR{\mathcal{R}}

\def\cB{\mathcal{B}}

\def\cY{\mathcal{Y}}

\def\kS{\mathfrak{S}}

\def\bE{\mathbb{E}} 
\def\bN{\mathbb{N}}
\def\bP{\mathbb{P}}

\def\bR{\mathbb{R}}
\def\bZ{\mathbb{Z}}

 \def\Z{\bZ}  \def\R{\bR}\def\N{\bN}

\def\zvec{\mathbf{z}}

\def\w{\omega}

\def\m1{\mathbf{1}}

 \def\wt{\widetilde}    \def\wc{\widecheck}
 
\def\E{\bE}
\def\P{\bP} 

\def\OSP{(\Omega, \kS, \P)}

\def\funct lp{L} 
\def\funct lpbar{\bar L} 
\def\dr{\mathcal{DR}}

\def\Uset{\mathcal U}

\def\Gpp{G}

\def\og{{\preceq}}
\def\ogs{{\prec}}

\def\wg{{w}}

\DeclareMathOperator{\Var}{Var}

\DeclareMathOperator{\ri}{ri}    

\newcommand{\argmax}[1]{\underset{#1}{\arg\max}\;}

\newcommand{\rhodown}[1]{{\underline{#1\mkern-3mu}\mkern4mu}}
\newcommand{\rhoup}[1]{{\bar{#1}}}

\definecolor{darkgreen}{rgb}{0.0,0.5,0.0}
\definecolor{darkblue}{rgb}{0.0,0.0,0.3}
\definecolor{nicosred}{rgb}{0.65,0.1,0.1}
\definecolor{light-gray}{gray}{0.7}
\allowdisplaybreaks[1]

\usepackage{filemod}

\def\cif1{v}   

\def\deq{\overset{d}=}

   \def\tincr{t}   
\def\Xw{X}  
\def\Yw{\omega}  
   
\def\Rb{\mathcal{R}}
\def\Sf{\mathcal{F}}

\def\pr{\mathcal{P}}
\def\st{\mathcal{S}}
\def\rec{\bold{R}}
\def\pc{p_c}
\def\Ra{\mathfrak{R}_\alpha}
\def\rim#1{\widehat{#1}} 
\def\ct{o}
\def\Dop{D}    
\def\Rop{R}  
\def\Sop{S}    

\def\arr{a} 
\def\barr{b}
\def\serv{s}
\def\depa{d} 
\def\sojo{t}
\def\arrv{\mathbf\arr}  \def\barrv{\mathbf\barr} 
\def\servv{\mathbf\serv}
\def\depav{\mathbf\depa} 
\def\sojov{\mathbf\sojo}

\newcommand{\lzb}{\llbracket}   
\newcommand{\rzb}{\rrbracket}   

\newcommand\abullet{{\scaleobj{0.6}{\bullet}}}  
\newcommand\bbullet{{\raisebox{0.5pt}{\scaleobj{0.6}{\bullet}}}} 
\newcommand\brbullet{{\raisebox{-1pt}{\scaleobj{0.5}{\bullet}}}} 

\usepackage[nomessages]{fp}
\newcounter{usedm}
\newcounter{usedn}
\newcommand*{\lppwl}[6]{
	
	\FPeval{\m}{round(#5*(#3-#1)/(#3+#4-#1-#2),0)}
	
	\FPeval{\n}{round(#5*(#4-#2)/(#3+#4-#1-#2),0)}
	
	\FPeval{\stlength}{(#3+#4-#1-#2)/#5}
	
	\setcounter{usedm}{0}
	\setcounter{usedn}{0}
	
	\foreach\i in{1,...,{#5}}{
		\FPrandom{\x}
		
		\FPeval{\y}{(\m-\theusedm)/(\m+\n-\theusedm-\theusedn)}
		
		\FPeval{\startx}{#1+(\theusedm)*\stlength}
		\FPeval{\starty}{#2+(\theusedn)*\stlength}
		
		\ifthenelse{\lengthtest{\x pt > \y pt}}{
			\FPset{\endx}{\startx}
			\FPadd{\endy}{\starty}{\stlength}
			\stepcounter{usedn}
		}
		{
			\FPadd{\endx}{\startx}{\stlength}
			\FPset{\endy}{\starty}
			\stepcounter{usedm}
		}
		\draw[#6](\startx,\starty) -- (\endx,\endy);
	};
}

\begin{document}

\title{Local stationarity of exponential last passage percolation}

\author[M.~Bal\'azs]{M\'arton Bal\'azs}
\address{M\'arton Bal\'azs\\ University of Bristol\\  School of Mathematics\\ Fry Building\\ Woodland Rd.\\   Bristol BS8 1UG\\ UK.}
\email{m.balazs@bristol.ac.uk}
\urladdr{https://people.maths.bris.ac.uk/~mb13434/}
\thanks{M.\ Bal\'azs was partially supported by EPSRC's EP/R021449/1 Standard Grant.}

\author[O.~Busani]{Ofer Busani}
\address{Ofer Busani\\ University of Bristol\\  School of Mathematics\\ Fry Building\\ Woodland Rd.\\   Bristol BS8 1UG\\ UK.}
\email{o.busani@bristol.ac.uk}
\urladdr{https://people.maths.bris.ac.uk/~di18476/}
\thanks{O. Busani was supported by EPSRC's EP/R021449/1 Standard Grant.}

\author[T.~Sepp\"al\"ainen]{Timo Sepp\"al\"ainen}
\address{Timo Sepp\"al\"ainen\\ University of Wisconsin-Madison\\  Mathematics Department\\ Van Vleck Hall\\ 480 Lincoln Dr.\\   Madison WI 53706-1388\\ USA.}
\email{seppalai@math.wisc.edu}
\urladdr{http://www.math.wisc.edu/~seppalai}
\thanks{T.\ Sepp\"al\"ainen was partially supported by  National Science Foundation grant  DMS-1854619 and by  the Wisconsin Alumni Research Foundation.}

\keywords{local stationarity, coalescence, corner growth model, directed percolation, geodesic, random growth model, last-passage percolation, queues}
\subjclass[2010]{60K35, 60K37} 
\date{\today} 
\begin{abstract}
	We consider point to point last passage times to every vertex in a neighbourhood of size $\delta N^{\nicefrac{2}{3}}$, distance $N$ away from the starting point. The increments of these last passage times in this neighbourhood are shown to be \emph{jointly equal} to their stationary versions with high probability that depends on $\delta$ only. With the help of this result we  show that\\
	1) the $\text{Airy}_2$ process is locally close to a Brownian motion in total variation;\\
	2) the tree of point to point geodesics starting from every vertex in a box of side length $\delta N^{\nicefrac{2}{3}}$ going to a point at distance $N$ agree inside the box with the tree of infinite geodesics going in the same direction;\\
	3) two geodesics starting from $N^{\nicefrac{2}{3}}$ away from each other, to a point at distance $N$ will not coalesce too close to either endpoints on the macroscopic scale.\\
	Our main results rely on probabilistic methods only. 	   
\end{abstract}
\maketitle

\tableofcontents

	\newcommand{\geodesic}[3]
	{
	\begin{scope}[shift={#1},rotate=#2]
	\draw [scale=3](0,0)
	\foreach \x in {1,...,#3}
	{   to [out=90-20,in=180+20] ++(1/#3,1/#3) to [out=20,in=270-20] ++(1/#3,1/#3)
	};
	\end{scope}
	}
	
%


 \section{Introduction.}
  Last Passage Percolation (LPP) belongs to the KPZ universality class  where models of random surface growth exhibit height and transversal fluctuations exponent of order  $\nicefrac{1}{3}$  and $\nicefrac{2}{3}$ respectively. The different models in the KPZ universality class are believed to have the same limiting behaviour under this scaling. The LPP with exponential weights belongs to the set of models in the KPZ universality class that are exactly solvable, or integrable. For models in this group, one can obtain closed form expressions for  their prelimiting statistics. Coupling this with techniques from combinatorics, representation and random matrix theory, one can take the limits of the prelimiting  expression to obtain the statistics of the limiting object. By the KPZ universality conjecture, this should be the limit of all models in the KPZ universality class.
  
   One of the interesting questions about the model is its local prelimiting fluctuations. To make this more concrete, let $G_x$ be the last passage time between the points $(0,0)$ and $x$. Define
  \begin{align}\label{bc}
  L^N_{(x,y)}=G_{(N,N)+y}-G_{(N,N)+x}.
  \end{align}
  It is known that if $|x|,|y|=O(1)$, then $L^N$ should be close to a stationary cocycle called Busemann function \cite{geor-rass-sepp-17-buse}. In fact, these stationary cocycles are defined as, roughly speaking, the limits when $N$ is taken to infinity in \eqref{bc}. Busemann functions can be thought of as the extension of the stationary LPP to the whole lattice and play a major role in the study of infinite geodesics. 
  
  The main contribution of this paper is to show  the convergence in total variation of $L^N$ to the Busemann function when $|x|,|y|\leq \delta N^\frac{2}{3}$ and $\delta$ goes to $0$. Moreover, the results are quantitative; we show that the decay of the error is polynomial in $\delta$. We stress that this cannot be simply obtained by using the  'Crossing Lemma'. Indeed, in order to compare the LPP increments to those of the stationary LPP one must tweak the intensity of the stationary LPP by order of $N^{-\frac{1}{3}}$ such that the error of the approximation along each edge is of the order of  $N^{-\frac{1}{3}}$ as well. Therefore, a simple union bound on the different $N^\frac{2}{3}$ edges will give $N^\frac{2}{3}N^{-\frac{1}{3}}=N^\frac{1}{3}$ and will not work. In a recent work, Fan and Sepp{\"a}l{\"a}inen \cite{fan-sepp-arxiv} obtained a coupling of different Busemann functions using queueing mapping. We use  new insights on this coupling to obtain the result which we refer to as \textit{local stationarity}. The rest of our results are applications of local stationarity to questions about the $\text{Airy}_2$ process and geodesics. 
   
  LPP can be viewed as a $1+1$ dimensional growing surface, and also as a Markov process that takes values in the space of continuous functions. Using the $1:2:3$ KPZ scaling, the conjectural limit of this Markov process is believed to be the KPZ-fixed point \cite{matetski2016kpz}. An extension of this limiting object was shown to exist recently in \cite{dauvergne2018directed}. In \cite{joha-03} Johansson showed the convergence of the spatial fluctuations to the $\text{Airy}_2$ process  minus a parabola and that the limit is continuous. As was mentioned previously, the fact that LPP has stationary counterparts whose spatial fluctuations are that of a simple random walk suggests that locally, the $\text{Airy}_2$ process should have a Brownian behaviour around a fixed point. The existing results in the literature regarding the Brownian behaviour of the $\text{Airy}_2$ process can be roughly divided into two groups -
  \begin{enumerate}
  	\item \label{g1}  on a small interval $[0,\epsilon]$ the $\text{Airy}_2$ process should be close to the Brownian motion in some sense
  	\medskip
  	\item \label{g2} on the interval $[0,1]$ the law of the $\text{Airy}_2$ process can be related to that of the Brownian motion. 
  \end{enumerate}
	\medskip
   In Group (\ref{g1}),   Pimentel \cite{pime-16}, in the LPP setup, showed that locally the $\text{Airy}_2$ process converges weakly to a Brownian motion in the Skorohord topology. The proof relied on a technique called 'Comparison Lemma' or 'Crossing Lemma'. The idea is that the spatial increments of the last passage time can be compared with high probability to stationary increments with a small drift. In \cite[Theorem 4.14]{matetski2016kpz}, Matetski, Quastel and Remenik showed that the  $\text{Airy}_2$ process has Brownian regularity and converges to the two-sided brownian motion in finite dimensional distributions. In \cite{pimentel2019brownian} (where some of the results lie in Group (\ref{g2})) Pimentel extends the results in \cite{pime-16} while the convergence is still in the weak sense. In Group (\ref{g2}),   Corwin and Hammond \cite{corw-hamm-14} showed that the Airy line ensemble minus parabola, conditioned on its values at the boundaries has the distribution of Brownian bridges conditioned not to meet. Building on these ideas, Hammond obtained  through the Brownian LPP \cite{hammond2016brownian}, among other things, a control on the moment of the Radon-Nykodim derivative of the law of the Airy line ensemble with respect to the Brownian bridge and a modulus of continuity of the $\text{Airy}_2$ process (see also \cite{calvert2019brownian}). In LPP on the lattice, control on the modulus of continuity of the prelimiting spatial fluctuations was obtained in \cite{basu2018time} by Basu and Ganguly.  In \cite{dauvergne2018basic} Dauvergne and Vir\'ag, using better insight on the sampled  Airy line ensemble, managed to show that the Airy line ensemble can be approximated, in total variation, by Brownian bridges, conditioned on not intersecting, without the conditioning on the lower boundary that appears in the Brownian Gibbs property.  Our result is concerned in comparing the  $\text{Airy}_2$ process with a Brownian motion on a small interval. Our result on the Brownian regularity of $\text{Airy}_2$ process lies in Group (\ref{g1}). In Theorem \ref{thm:airy} we show that $\text{Airy}_2$ process is close to a Brownian motion of rate $2$ in total variation. This improves  similar results in Group (\ref{g1}). As a consequence we show in Corollary \ref{thm:airyr} that the regularity of the $\text{Airy}_2$ process cannot be better than that of the Brownian one. Note that Corollary \ref{thm:airyr} can also be deduced by \cite[Theorem 1.1]{calvert2019brownian}. 
  
  Next we apply local stationarity to  study two aspects of the behaviour of geodesics, their behaviour close to the end points  which we refer to as \emph{stabilization}, and the coalescence of point to point geodesics starting from two points whose distance scales with $N$. Let us start with the latter. In the past few years the study of coalescence of geodesics has gained focus. Methods for the study of geodesics of growth models can be traced back to Newman and co-authors in \cite{howa-newm-01,howa-newm-97,lice-newm-96,newm-icm-95} for First Passage Percolation (FPP), another random growth model believed to be in the KPZ universality class. These methods were then used by Ferrari and Pimentel \cite{ferr-pime-05} and Coupier \cite{coup-11} to show that in LPP, for a fixed direction, from any point in the lattice there exists a.s.\ a unique infinite geodesic and that these geodesics coalesce. A first quantitive result on the coalescence of geodesics in LPP came from Pimentel \cite{pime-16}, who showed that two infinite geodesics with the same direction, coming out of two points that are $k$ away from each other will coalesce after about $k^{\nicefrac{3}{2}}$ steps. The tail of the decay was conjectured to be of exponent $-\nicefrac{2}{3}$. The proof used the fact that the  geodesic tree has the same distribution as its dual tree and existing bounds on the distribution of exit point of a geodesic of stationary LPP. The question of showing that the geodesics will not coalesce too far compared to $k$ i.e.\ a matching upper bound, was left open. This question was then taken up by Basu, Sarkar and Sly \cite{basu-sark-sly-arxiv-17} who proved the -$\nicefrac{2}{3}$ exponent for the lower bound and a matching upper bound. In that paper, the authors also proved  a polynomial upper bound for  point to point coalescence. In \cite{sepp2019coal} Sepp{\"a}l{\"a}inen and Shen, studied coalescence of infinite geodesics. Without relying on  integrable probability methods, they proved the upper bound  and a new  exponential  lower bound for fast coalescence of the geodesics. In \cite{zhang2019optimal} Zhang proved the optimal bounds of $-\nicefrac{2}{3}$ for point to point coalescence of two geodesics leaving from two points of fixed distance $k$. The proof relies on diffusive concentration of geodesics fluctuations coming from integrable probability.
   
  In this work, we also study the coalescence of two geodesics starting from two points whose distance scales with the length of the geodesics.  Results of that flavour were proved  in \cite{hammond2020exponents} and \cite{basu2019fractal} for Brownian LPP, and in \cite{FerrSpohn2003} for Poissonian LPP. More precisely, we are interested in the following question; if $\pi^1$ and $\pi^2$ are the geodesics starting from $(0,0)$ and $(0,N^\frac{2}{3})$ respectively,  terminating at $(N,N)$, what is the typical distance of the coalescence point from the three endpoints? We show in Theorem \ref{thm:ubc} and Theorem \ref{thm:ubc2} that the coalescence point will not be too close, on a macroscopic scale, to any of the end points. We emphasize that the methods used in \cite{pime-16} and \cite{sepp2019coal} cannot be used here, as they rely on a well understood  duality principle for stationary LPP geodesics \cite{sepp-arxiv-18}.
    
  Let us now turn to stabilization. Let $\pi$ be the geodesic going from $(0,0)$ to the point $(N,N)$. Since the work of Johansson in \cite{joha-ptrf-00} it is known that the fluctuations of $\pi$ around the diagonal at any macroscopic point should be of order $N^{\nicefrac{2}{3}}$. If $1 \leq l << N$, as the geodesic is expected to have a self-similarity property, one would expect the fluctuation of $\pi$ in a square of size $l^2$ around the origin to be of order $l^{\nicefrac{2}{3}}$. A proof of this was given in \cite[Theorem 3]{basu-sark-sly-arxiv-17} with diffusive concentration bounds. In \cite{hammond2016brownian} Hammond considered the regularity of the spatial fluctuation around the point $(l,l)$ for the Brownian LPP while for the Corner Growth Model with exponential weights this was proven in \cite[Theorem 3]{basu2018time} by Basu and Ganguly. 
  The behaviour of infinite geodesics is somewhat better understood. This is due to the fact that the Busemann functions 'point out' the way in which the geodesic go, through the minimum gradient principle \cite{sepp-arxiv-18,geor-rass-sepp-17-geod,geor-rass-sepp-17-buse}. This implies that a link between point to point geodesics and infinite ones should provide better insight on the former. Consider a small square of side length $M$ around the origin. From each point in the square leaves  a unique geodesic that terminate at the point $(N,N)$. Let us denote by $\mathcal{T}^{pp}$ the tree consisting of all the geodesics starting from the square and ending at the point $(N,N)$. Similarly let $\mathcal{T}^{\infty}$ be the tree that consists of all infinite geodesics of direction $45^\circ$ starting from the square. Our stabilization result, Theorem \ref{thm:stb}, shows that on a square of side $M=\delta N^\frac{2}{3}$, the trees $\mathcal{T}^{pp}$ and $\mathcal{T}^{\infty}$ agree outside a set of probability of order power of $\delta$. We use this to show in \eqref{cor:stp}, for example, that the fluctuations of the point to point geodesic  in a small box of side $l$ around the origin are, with high probability, the same as those of a stationary geodesic for which the fluctuations are known to be of the order $l^\frac{2}{3}$.  Finally, we use stabilization to study coalescence of point to point geodesics where the distance of the starting point is fixed. More precisely, for fixed $k>0$ let $\pi^1$ and $\pi^2$ be the geodesics starting from $(0,0)$ and $(0,k^\frac23)$ respectively. Let $u^N$ be the coalescence point of $\pi^1$ and $\pi^2$. Similarly let $v^*$ be the coalescence point of two infinite geodesics starting at the points $(0,0)$ and $(0,k^\frac23)$ in direction $(1,1)$. Theorem \ref{thm:stpp} shows that $u^N$ converges weakly to $v^*$. In particular, using the results in \cite{basu-sark-sly-arxiv-17}, we show that the exponent for the decay of the tail of the distance of $u^N$ from the origin is $-\nicefrac{2}{3}$  as was shown in \cite{zhang2019optimal}.
  
  The main body of our arguments only uses probabilistic methods. The only integrable-probability input we use is the emergence of the $\text{Airy}_2$ process as the limit of the increments of the last passage time.
  \vspace*{4px}\\
  {\bf Some general notation and terminology} 
  $\Z_{\ge0}=\{0,1,2,3, \dotsc\}$ and $\Z_{>0}=\{1,2,3,\dotsc\}$.  For $n\in\Z_{>0}$ we abbreviate  $[n]=\{1,2,\dotsc,n\}$.   A sequence of $n$   points  is denoted by $x_{0,n}=(x_k)_{k=0}^n=\{x_0,x_1,\dotsc,x_n\}$, and in case it is a path of length $n$  also by $x_{\bbullet}$.   $a\vee b=\max\{a,b\}$.     $C$ is a constant whose value can change from line to line.

  The standard basis vectors of $\R^2$ are  $e_1=(1,0)$ and $e_2=(0,1)$.  For a point  $x=(x_1,x_2)\in\R^2$  the $\ell^1$-norm  is   $\abs{x}=\abs{x_1} + \abs{x_2}$ .     We call the $x$-axis occasionally the $e_1$-axis, and similarly the $y$-axis and the $e_2$-axis are the same thing.  
  Inequalities on $\R^2$ are interpreted coordinatewise: for $x=(x_1,x_2)\in\R^2$ and $y=(y_1,y_2)\in\R^2$,   $x\le y$  means $x_1\le y_1$ and $x_2\le y_2$.    Notation $[x,y]$ represents both   the line segment $[x,y]=\{tx+(1-t)y: 0\le t\le 1\}$ for $x,y\in\R$ and the rectangle  $[x,y]=\{(z_1,z_2)\in\R^2:   x_i\le z_i\le  y_i \text{ for }i=1,2\}$ for $x=(x_1,x_2), y=(y_1,y_2)\in\R^2$. The context will make clear which case is used.   $0$ denotes the origin of both $\R$ and $\R^2$. $X\sim$ Exp($\lambda$) for $0<\lambda<\infty$ means that random variable $X$ has exponential distribution with rate $\lambda$, in other words $P(X>t)=e^{-\lambda t}$ for $t\ge 0$.  The mean is $E(X)=\lambda^{-1}$ and variance $\Var(X)=\lambda^{-2}$. In general,  $\overline X=X-EX$ denotes a random variable $X$ centered at its mean. If $x<y$ we write $\lzb x,y \rzb$ for the set of integers $[x,y]\cap \mathbb{Z}$. If $x,y\in \mathbb{R}^2$ such that $x\leq y$ we denote by $\lzb x,y \rzb =[x,y]\cap \mathbb{Z}^2$. If $A\subset \Z^2$ is connected, we let $\mathcal{E}(A)$ denote the set of edges induced by $A$ in $\Z^2$. 
 	
\section{Main results}
Let $\omega=\{\omega_x\}_{x\in \mathbb{Z}^2}$ be a set of random weights on the vertices of $\mathbb{Z}^2$. We assume that $\omega$ is i.i.d.\ of $\text{Exp}(1)$ distribution. For $o\in \Z^2$, we define the last-passage time on $o+\Z_{\geq0}^2$  to be 
\be\label{v:G}  
\Gpp_{o,y}=\max_{x_{\brbullet}\,\in\,\Pi_{o,y}}\sum_{k=0}^{\abs{y-o}_1}\w_{x_k}\quad\text{ for }  y\in o+\Z_{\geq 0}^2. 
\ee
$\Pi_{o,y}$ is the set of   paths $x_{\bbullet}=(x_k)_{k=0}^n$  that start at  $x_0=o$,  end at $x_n=y$ with $n=\abs{y-o}_1$,  and have increments $x_{k+1}-x_k\in\{e_1,e_2\}$. The a.s.\ unique path $\pi^{o,y}\in \Pi_{o,y}$ that attains the maximum in \eqref{v:G} is called the geodesic from $o$ to $y$. Similarly we  define the stationary LPP (see \eqref{Gr2}) $G^{\frac{1}{2}}_{o,y}$ associated with the direction $(1,1)$.  
 Let $\rim{R}^c=[N\xi-cN^\frac{2}{3}\xi,N\xi]$ be the rectangle whose lower left corner is $(N-cN^\frac{2}{3},N-cN^\frac{2}{3})$ and whose upper right corner is $(N,N)$. Let $\mathcal{E}(\rim{R}^c)$ be the set of directed edges in the subgraph of $\Z^2$ induced by the vertices in $\rim{R}^c$. We define the following random variables indexed by $\mathcal{E}(\rim{R}^c)$
 \begin{align*}
 	H^{N,c}_{(x,y)}&=G_{o,y}-G_{o,x} \quad (x,y)\in \mathcal{E}(\rim{R}^c)\\
 	H^{\frac{1}{2},N,c}_{(x,y)}&=G^{\frac{1}{2}}_{o,y}-G^{\frac{1}{2}}_{o,x},
 \end{align*}
 and 
 \begin{align*}
 	H^{N,c}&=\Big\{H^{N,c}_{(x,y)}\ :\ (x,y)\in \mathcal{E}(\rim{R}^c)\Big\}\\
 	H^{\frac{1}{2},N,c}&=\Big\{H^{\frac{1}{2},N,c}_{(x,y)}\ :\ (x,y)\in \mathcal{E}(\rim{R}^c)\Big\}.
 \end{align*}
 Let $d_\text{TV}(\cdot,\cdot)$ denote the total variation distance between two distributions. If $X\sim\mu$ and $Y\sim \nu$, we abuse notation and write $d_\text{TV}(X,Y)$ for $d_\text{TV}(\mu,\nu)$. The following is the main result of the paper. It shows that on the scale of $N^\frac{2}{3}$, around the point $(N,N)$, local increments of $G$ {\bf jointly equal} to those in $G^{\frac{1}{2}}$ with high probability. The choice of $\frac{1}{2}$ is for a neater exposition of the result, our proof is for every $0<\rho<1$, where the constants depend on $\rho$.
  \begin{theorem}\label{thm:loc}
 	There exists $c_0>0$ and $C(c_0)>0$, such that for $c\leq c_0$ and $N\geq1$
 	\begin{align}\label{loc}
 		d_\text{TV}\Big(H^{N,c},H^{\frac{1}{2},N,c}\Big)\leq Cc^\frac{3}{8}.
 	\end{align}
 \end{theorem}
 Let 
 \begin{align*}
 	L^N_x=2^{-\frac{4}{3}}N^{-\frac{1}{3}}\Big(G_{(0,0),(N+x(2N)^\frac{2}{3},N-x(2N)^\frac{2}{3})}-4N\Big).
 \end{align*}
 It is known \cite{BoroFerr2008} that  $L^N_x\rightarrow\mathcal{A}_2(x)-x^2$ as $N\rightarrow \infty$, where $\mathcal{A}_2(x)$ is the $\text{Airy}_2$ process and the convergence is in distribution, in the topology of continuous functions on compact sets. Set
 \begin{align*}
 	\mathcal{A}'_2(x)=\mathcal{A}_2(x)-\mathcal{A}_2(0)-x^2.
 \end{align*}
  Let $\mathcal{B}$ be an two-sided Brownian motion of variance $2$ on $\R$. Our next result shows that locally, the $\text{Airy}_2$ process looks like a Brownian motion in a strong sense. It  follows easily from Theorem \ref{thm:loc} (see a proof in the end of Section \ref{sec:stb}).
 \begin{theorem}\label{thm:airy}
 	There exists $c_0>0$ and $C(c_0)>0$, such that for $c\leq c_0$
 	\begin{align*}
 		d_\text{TV}\Big(\mathcal{A}'_2|_{[-c,c]},\mathcal{B}|_{[-c,c]}\Big)\leq Cc^\frac{3}{8}. 
 	\end{align*}
 \end{theorem}
Let $I\subset \R$ be an interval and let 
\begin{align*}
	\omega_B(t)=2\sqrt{t\log(t)^{-1}}
\end{align*}
be the modulus of continuity of the Brownian motion. In \cite[Theorem 1.11]{hammond2016brownian} Hammond showed that the regularity of the Airy process is not worse than that of a Brownian motion i.e.\ 
 \begin{align*}
 	\sup_{t\in I}\limsup_{h\downarrow 0}\frac{\mathcal{A}_2(t+h)-\mathcal{A}_2(t)}{\omega_B(h)}< \infty\quad \text{with probability $1$}.
 \end{align*}
 As a corollary of Theorem \ref{thm:airy}, we show that the regularity of the $\text{Airy}_2$ process is not better than that of a Brownian motion.
 \begin{corollary}\label{thm:airyr}
 	\begin{align*}
 		\sup_{t\in I}\limsup_{h\downarrow 0}\frac{\mathcal{A}_2(t+h)-\mathcal{A}_2(t)}{\omega_B(h)}\geq 1\quad \text{with probability $1$}.
 	\end{align*} 
 \end{corollary}
 Let us now turn to our stabilization results. The set of possible asymptotic velocities or direction vectors for semi-infinite up-right paths is $\Uset=\{(t, 1-t): 0\le t\le 1\}$, 
 with relative interior $\ri\Uset=\{(t, 1-t): 0< t<1\}$. For $\xi \in \ri \Uset$, let $\Rb^{\xi,N}=[0,N\xi]$ be the rectangle whose lower left corner is $(0,0)$ and upper right corner is $N\xi$. Let $\pi$ be an up-right path whose origin is $(0,0)$. Let $I^\pi=\{i:\pi_i\in \Rb^{\xi,N}\}$ be the set of  indices of $\pi$ for which $\pi$ is in $\Rb^{\xi,N}$. We define $\pr^{\xi,N}(\pi)$ to be the restriction of the path $\pi$ on the rectangle $\Rb^{\xi,N}$, that is, $\pr^{\xi,N}(\pi)$ is a finite path defined by
\begin{align}
(\pr^{\xi,N}(\pi))_i=\pi_i \quad \forall i\in I^\pi.
\end{align}
Let $\pi^{x,\xi\infty}$ be the infinite geodesic starting from $x$ whose direction is $\xi$ \cite{sepp-arxiv-18}. For $M<N$, define the following event
\begin{align}\label{Sf}
\st^{\xi,M}=\{\pr^{\xi,M}(\pi^{x,\xi\infty})=\pr^{\xi,M}(\pi^{x,\xi N})\text{ for all } x\in \Rb^{\xi,M}\}.
\end{align}
$\st^{\xi,M}$ is the event on which any geodesic leaving from any site  $x\in \Rb^{\xi,M}$  and terminating at $\xi N$ agree with the infinite geodesic $\pi^{x,\xi\infty}$ on $\Rb^{\xi,M}$ (see Figure \ref{fig:stb}).  Our first result gives a lower bound on the probability of stabilization on small enough rectangles.
\begin{theorem}\label{thm:stb}
	Let $\xi\in\ri \Uset$ and $c>0$. For any $M > 0$ such that $M\leq cN^\frac{2}{3}$,  there exists $C(\xi,c)>0$, locally bounded in $c$, such that
	\begin{align}
	\P(\st^{\xi,M})\geq 1-CN^{-\frac{1}{4}}M^\frac{3}{8}.
	\end{align}
\end{theorem}
Define the the following set
\begin{align}
\cF=\{f\in \R_+\rightarrow\R_+:\text{$f$ is increasing and $f(t)\leq t$}\}.
\end{align}
For $f\in \Sf$, we say the sequence $\rec^{\xi,f}=\{\Rb^{\xi,f(N)}\}_{N\in\Z_{>0}}$ stabilizes if 
\begin{align}
\lim_{N\rightarrow \infty}\P(\st^{\xi,f(N)})=1.
\end{align} 
In words, the sequence $\rec^{\xi,f}$ stablizes if with high probability the tree of all the geodesics starting at points in $\Rb^{\xi,f(N)}$ and terminating at $\xi N$ agree on $\Rb^{\xi,f(N)}$ with the tree of infinite geodesics in direction $\xi $ starting from $\Rb^{\xi,f(N)}$. As $f$ is a function of $N$ we shall often write $f$ instead of $f(N)$ so that $\Rb^{\xi,f}=\Rb^{\xi,f(N)}$. As a corollary of Theorem \ref{thm:stb} we have the following.
\begin{corollary}\label{cor:stb}
		For any $\xi\in\ri \Uset$ and $f\in \Sf$ such that $f(t)=o(t^\frac{2}{3})$, $\rec^{\xi,f}$ stabilizes and there exists $C(\xi)>0$ such that
	\begin{align}
	\P(\st^{\xi,f})\geq 1-CN^{-\frac{1}{4}}f(N)^\frac{3}{8}
	\end{align}
\end{corollary}
\begin{figure}[t]
	\includegraphics[scale=1]{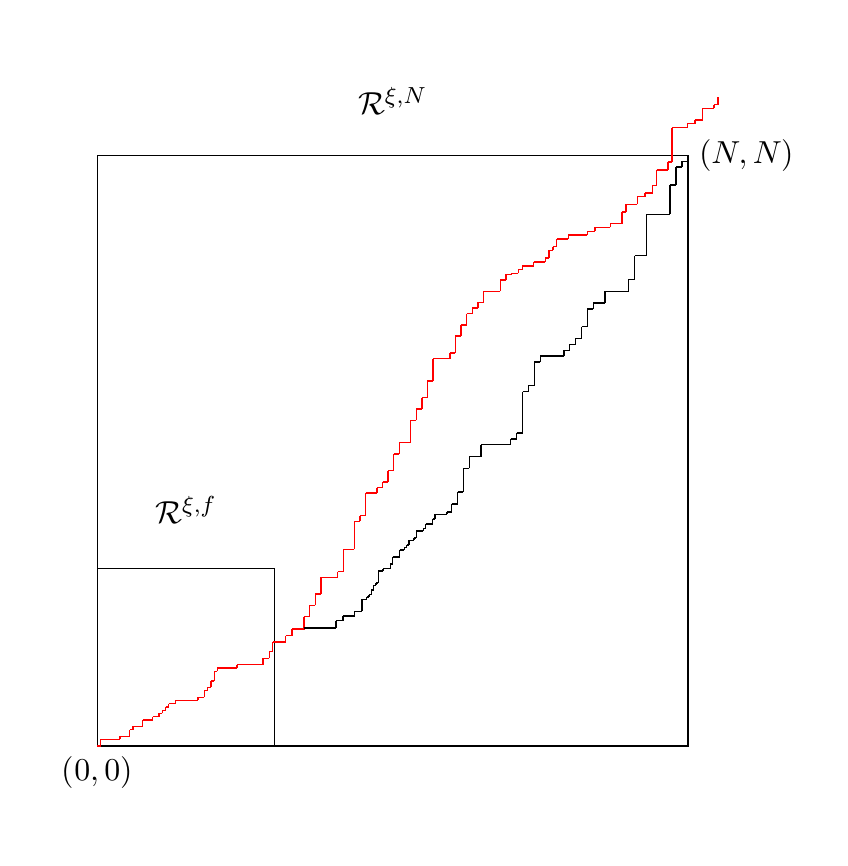}
	\caption{\small The infinite geodesic $\pi^{0,\xi \infty}$ and the geodesic $\pi^{0,\xi N}$ agree in the box $\Rb^{\xi,f}$. On the event $\st^{\xi,f}$ for any $x\in \cR^{\xi,f}$ the geodesics $\pi^{x,\xi \infty}$ and $\pi^{x,\xi N}$ have the same restriction on the small square $\cR^{\xi,f}$.}\label{fig:stb}
\end{figure}
\medskip
As was mentioned earlier, stabilization can be used to study the behaviour of point to point geodesics close to their end points. For Fixed $M$, on $\Rb^{\xi,M}$ consider the stationary LPP $\rim{G}^{\rho}_{\xi (M+1),0}$ (see \eqref{Gr12}), starting from the point $\xi (N+1)$ and terminating at the origin. Let us denote its geodesic by $\rim{\pi}^\rho$. Let $\pi^{0,\xi N}$ be the geodesic of LPP starting from the origin $0$ and terminating at $\xi N$. The following corollary relates the behaviour of $\pi^{0,\xi N}$ to that of $\rim{\pi}^\rho$. 
\begin{corollary}\label{cor:stp}
	For fixed $M\in\Z_{> 0}$
	\begin{align}\label{stp}
		\lim_{N\rightarrow \infty}d_{TV}\big(\rim{\pi}^\rho,\pr^{\xi,M}(\pi^{0,\xi N})\big)=0.
	\end{align}
	In particular, there exists $C(\xi)>0$ such that for $l\in\Z_{> 0}$
	\begin{align}\label{stp1}
		\lim_{N\rightarrow \infty}\P(|\pi^{0,\xi N}_2l-(l,l)|>rl^\frac{2}{3})\leq Cr^{-3}.
	\end{align}
\end{corollary}
\begin{proof}
	\eqref{stp} follows from Corollary \ref{cor:stb} and the fact that, the distribution of an infinite geodesic going backwards is that of a stationary one (see the proof of Theorem \ref{thm:stb}). \eqref{stp1} follows from \eqref{stp} and well known bounds on the fluctuations of stationary geodesics \cite[Theorem 5.3]{sepp-cgm-18}. 
\end{proof}
Stabilization can help relating results on infinite geodesics to results on point-to-point geodesics and vice versa. Consider the points $v_1=(0,0)$ and $v_2=k^{\nicefrac{2}{3}}e_2$ for some $k\geq1$. Let $\pi^{v_1,\xi \infty}$ and $\pi^{v_2,\xi \infty}$ be the infinite geodesics in direction $\xi$ starting from $v_1$ and $v_2$ respectively. Let $v^*=(v^*_1,v^*_2)$  be the point in $\pi^{v_1,\xi \infty}\cap \pi^{v_2,\xi \infty}$ that is closest to the origin. Similarly let $u^N$ be the closest point in $\pi^{v_1,\xi N}\cap \pi^{v_2,\xi N}$ to the origin. In \cite{basu-sark-sly-arxiv-17} Basu, Sarkar and Sly showed that there exist universal constants $C_1,C_2,R_0$ such that for every $k>0$ and $R>R_0$ 
\begin{align}\label{bss3}
	C_1R^{-\frac{2}{3}}\leq \P(|v^*|>Rk) \leq C_2R^{-\frac{2}{3}}.
\end{align}
Moreover, they showed that there exist $C,R_0,c>0$ such that for every $k>0$ and $R>R_0$
\begin{align}\label{bss}
	\limsup_{n\rightarrow \infty} \P(u_1^N>Rk)\leq CR^{-c}.
\end{align}
The exponent $c$ in \eqref{bss} was not identified but was conjectured to be $\nicefrac{2}{3}$. This was recently settled by Zhang in \cite{zhang2019optimal} using input from integrable probability. We now show how this can be approached via our stabilization result.
\begin{theorem}\label{thm:stpp}
	The sequence $|u^{N}|$ converges weakly to $|v^*|$. Moreover, there exist universal constants $C_1,C_2,R_0>0$ such that for $R>R_0$, for any $k\geq 1$ and $N>(Rk)^5$
	\begin{align*}
	C_1R^{-\frac{2}{3}}\leq \P(|u^N|>Rk) \leq C_2R^{-\frac{2}{3}}.
	\end{align*}
\end{theorem}
\begin{proof}
	 If exactly one of the events $\{|v^*|>Rk\},\{|u^N|>Rk\}$ occurs then paths must not have coalesced in $\Rb^{\xi,\frac12 Rk}$, in other words, $\st^{\xi,\frac12 Rk}$ does not occur. Therefore, via the symmetric difference and using Theorem \ref{thm:stb},
	 \begin{align*}
	 	|\P(|v^*|>Rk)-\P(|u^N|>Rk)|\leq \P\big(\{|v^*|>Rk\}\Delta\{|u^N|>Rk\}\big)\leq \P\big((\st^{\xi,\frac12 Rk})^\text{c}\big)\leq N^{-\frac{1}{4}}(Rk)^{\frac{3}{8}},
	 \end{align*}
	 which shows that  $|v^{N}|$ converges weakly to $|v^*|$. 
	 Taking $N=(Rk)^{5}$ and using Theorem \ref{thm:stb}  
	\begin{align}
		\P(|v^*|>Rk)- R^{-\frac{7}{8}}\leq \P(|u^N|>Rk)\leq  \P(|v^*|>Rk)+ R^{-\frac{7}{8}}\label{bss2}. 
	\end{align}
	As $\nicefrac{7}{8}> \nicefrac{2}{3}$ \eqref{bss2} and \eqref{bss3} imply the result.
\end{proof}
\medskip

Let us now turn to our coalescence results. In $\Rb^{\xi, N}$, consider the points $\ct=\xi N$, $q^1=(0,0)$ and $q^2=aN^\frac{2}{3}e_2$ where $a>0$ and where we assume that $N$ is large enough so that $q^2\in \Rb^{\xi, N}$. Let 
\begin{align}
\cC^{a,\xi}=\pi^{q^1,o}\cap \pi^{q^2,o},
\end{align} 
be the points shared by the geodesics starting from $q_i$ and terminating at $\ct$ for $i\in {1,2}$. We define the coalescence point $\pc$ to be the unique point such that
\begin{align}
\pc\in \cC^{a,\xi} \quad \text{and} \quad \pc\leq x \quad \forall x\in \cC^{a,\xi},
\end{align}
as in Figure \ref{fig:coal}. Our next result shows that the point $\pc$ is not likely to be too close to the point $\ct$ on a macroscopic scale.
\begin{theorem}\label{thm:ubc}
	For every $a>0$ and $\xi\in\ri\Uset$, there exists a constant $C(\xi,a)>0$, locally bounded in $a$, such that for every $0<\alpha<1$ and $N>N(\alpha)$
	\begin{align}\label{cub}		
	\P(|\ct-\pc|\leq \alpha N)\leq C\alpha^\frac{2}{9}.
	\end{align}
\end{theorem}
The following result shows that the geodesics $\pi_{q^1,o}$ and $\pi_{q^2,o}$ do not coalesce too close to their origins on a macroscopic scale. Although the proof does not require local stationarity, we state it for completeness.
\begin{theorem}\label{thm:ubc2}
	For every $a>0$ and $\xi\in\ri\Uset$, there exists a constants $C(\xi,a)>0$ such that for every $0<\alpha<1$ and $N>N(\alpha)$ 
	\begin{align*}
	\P(|q^2-\pc|\leq \alpha N)\leq C\alpha^2.
	\end{align*}
\end{theorem} 
\begin{figure}[t]
	\includegraphics[scale=1]{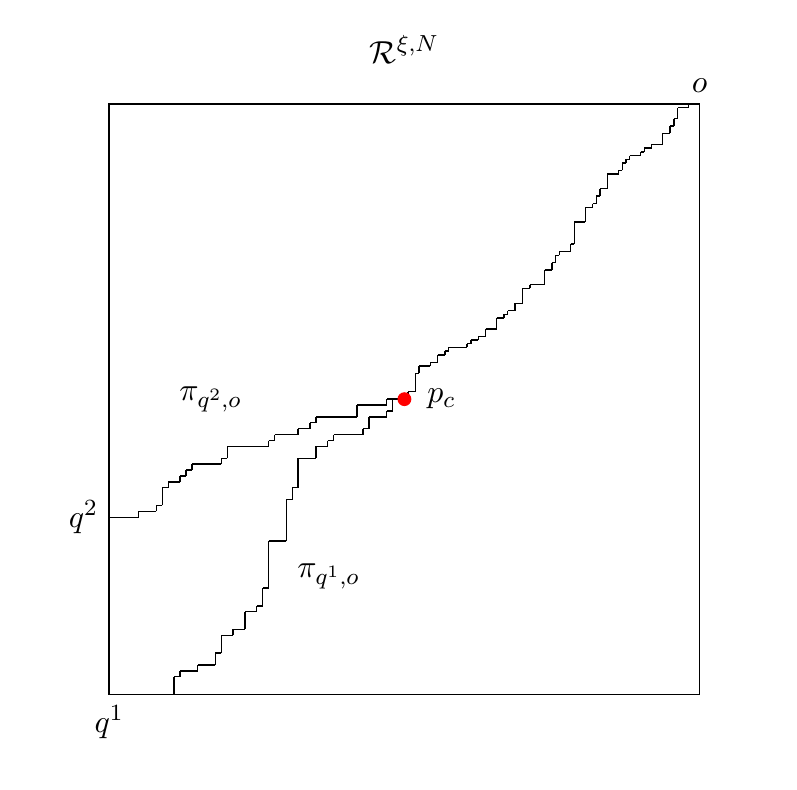}
	\caption{\small Two geodesics leaving from two points that are $aN^\frac{2}{3}$ far from one another, meet at the point $\pc$ (red). With high probability the point $\pc$ is not too close to the points $q^1,q^2,o$ on a macroscopic scale}\label{fig:coal}
\end{figure}
\subsection{Acknowledgements}  The authors thank B\'alint T\'oth for useful discussions and comments and Alan Hammond for guidance to the literature. 
\section{Preliminaries}	
	  \medskip{\bf Order on Geodesics}
	  \\ We would like to construct a partial order  on the set of non-intersecting paths in $\Z^2$. For $x,y\in \Z^2$ we write $x \preceq y$ if $y$ is below and to the right of $x$, i.e.
	  \begin{align}
	  	x_1 \leq y_1 \quad \text{and} \quad x_2\geq y_2.
	  \end{align}
	  We also write $x\ogs y$ if
	  \begin{align}
	  	x \og y \quad \text{and} \quad x\neq y
	  \end{align}
	  If $A,B\subset \Z^2$, we write $A \preceq B$ if 
	  \begin{align}
	  	x \preceq y \quad \forall x\in A,y\in B.
	  \end{align}
	  A down-right path is a bi-infinite sequence $\cY=(y_k)_{k\in\Z}$  in   $\Z^2$ such that    $y_k-y_{k-1}\in\{e_1,-e_2\}$ for all $k\in\Z$. Let $\dr$ be the set of infinite down-right paths in $\Z^2$. Let $\gamma_1,\gamma_2$ be two up-right paths in $\Z^2$ we write $\gamma_1\og \gamma_2$ if 
	  \begin{align}\label{og}
	  	\gamma_1\cap \cY \preceq \gamma_2\cap\cY \quad \forall\cY\in \dr,
	  \end{align}
	  where we assume the inequality to be vacuously true if one of the intersections in \eqref{og} is empty(see Figure \ref{fig:ogs}).
	  \begin{figure}[t]
	  	\includegraphics[scale=1]{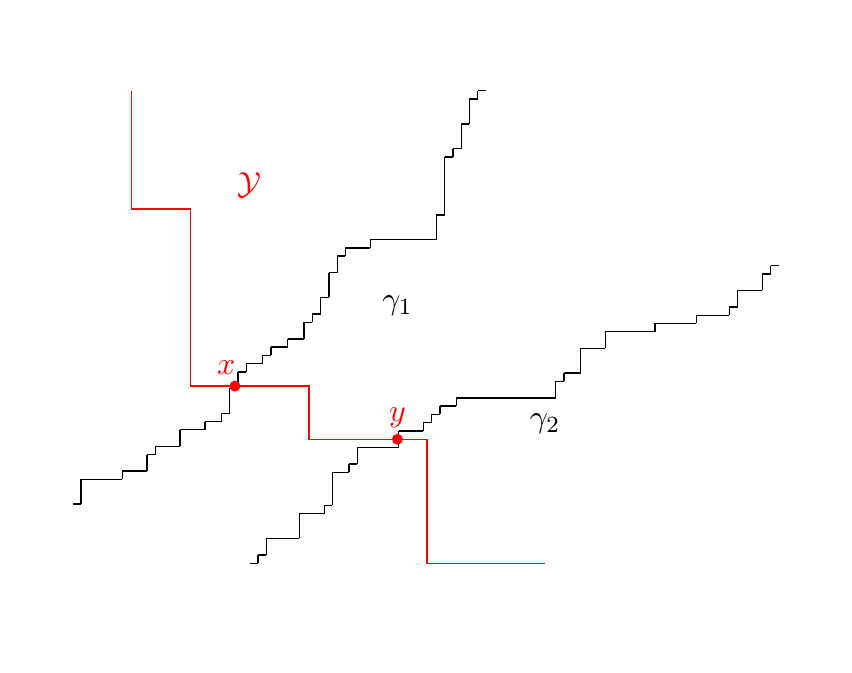}
	  	\caption{\small The two geodesics $\gamma_1$ and $\gamma_2$ are ordered i.e. $\gamma_1\ogs \gamma_2$. For any down-right path $\cY$ in $\Z^2$ the set of points $x=\cY\cap \gamma_1$ and  $y=\cY\cap \gamma_2$ are ordered, i.e. $x\ogs y$.}\label{fig:ogs}
	  \end{figure}
	
	  \medskip{\bf Stationary LPP} 
\\
	For each $o=(o_1,o_2)\in \Z^2$ and a parameter value $\rho\in(0,1)$ we introduce  the stationary  last-passage percolation process  $G^\rho_{o,\abullet}$ on $o+\Z_{\ge0}^2$.  This process has boundary conditions given by two independent sequences   
	\begin{align}\label{IJ}
		\{I^\rho_{o+ie_1}\}_{i=1}^{\infty} \quad\text{and}\quad 
		\{J^\rho_{o+je_2}\}_{j=1}^{\infty} 
			\end{align}
	 of i.i.d.\ random variables with marginal distributions $I^\rho_{o+e_1}\sim\text{Exp}(1-\rho)$ and $J^\rho_{o+e_2}\sim\text{Exp}(\rho)$.  Put  $G^\rho_{o,o}=0$ and on the boundaries 
	\be\label{Gr1} G^\rho_{o,\,o+\,ke_1}=\sum_{i=1}^k I_{ie_1} 
	\quad\text{and}\quad
	G^\rho_{o,\,o+\,le_2}= \sum_{j=1}^l  J_{je_2} .   \ee 
	Then in the bulk 
	for $x=(x_1,x_2)\in o+ \Z_{>0}^2$, 
	\be\label{Gr2}
	G^\rho_{o,\,x}= \max_{1\le k\le x_1-o_1} \;  \Bigl\{  \;\sum_{i=1}^k I_{o+ie_1}  + G_{o+ke_1+e_2, \,x} \Bigr\}  
	\bigvee
	\max_{1\le \ell\le x_2-o_2}\; \Bigl\{  \;\sum_{j=1}^\ell  J_{o+je_2}  + G_{o+\ell e_2+e_1, \,x} \Bigr\} .
	\ee
	For a  northeast endpoint  $p\in o+\Z_{>0}^2$, let $Z^\rho_{o,p}$  be the signed exit point of the geodesic $\pi^{o,p}_\bbullet$ of $G^\rho_{o,p}$ from the west and south boundaries  of $o+\Z_{>0}^2$. More precisely,
	\begin{align}\label{exit2}
	Z^\rho_{o,p}=
	\begin{cases}
	\argmax{k} \bigl\{ \,\sum_{i=1}^k I_{o+ie_1}  + G_{o+ke_1+e_2, \,x} \bigr\},  &\text{if } \pi^{o,p}_1=o+e_1,\\
	-\argmax{\ell}\bigl\{  \;\sum_{j=1}^\ell  J_{o+je_2}  + G_{o+\ell e_2+e_1, \,x} \bigr\},  &\text{if } \pi_1^{o,p}=o+e_2.
	\end{cases}
	\end{align}
 The value $G^\rho_{o,x}$ can be determined by \eqref{Gr1} and the following recursive relation
 \begin{align}\label{recu}
 	G^\rho_{o,x}=\omega_x+G^\rho_{o,x-e_1}\vee G^\rho_{o,x-e_2}.
 \end{align}
 Relation \eqref{recu} implies that one can backtrack the geodesic $\pi^{o,p}$ in the box $[o+e_1+e_2,p]$ in the following way; for each (directed) edge $(x,y)$ in $[o+e_1+e_2,p]$ assign the weight $\wg_{x,y}=G^\rho_{o,y}-G^\rho_{o,x}$. Let $m=|p-o|$, and denote $p_i=\pi^{o,p}_{i}$. We have
 \begin{align}\label{recu2}
 	&p_m=p,\\
 	&p_i=
 	\begin{cases}
 	p_{i+1}-e_1 &\mbox{if } \wg_{p_{i+1}-e_1,p_{i+1}}<\wg_{p_{i+1}-e_2,p_{i+1}} \\ \nonumber
 	p_{i+1}-e_2 & \mbox{if } \wg_{p_{i+1}-e_1,p_{i+1}}>\wg_{p_{i+1}-e_2,p_{i+1}}\end{cases}\quad |Z^\rho_{o,p}| \leq i\leq m-1.
 \end{align}
 In other words, we trace the geodesic $\pi^{o,p}$ backwards up to the exit point from the boundaries, by following the edges on which the increments of the process $G^\rho_{o,p}$ are minimal.\\ Next we consider LPP maximizing down-left paths. For $y\leq o$, define
 	\be\label{v:Gr}  
 \rim{G}_{o,y}=G_{y,o}. 
 \ee
  For each $o\in\Z^2$ and a parameter value $\rho\in(0,1)$ define a stationary last-passage percolation processes   $\rim{G}^\rho$ on $o+\mathbb{Z}^2_{\le0}$, with boundary variables  on the north and east,   in the following way. Let
\begin{align}\label{IJ hat}
&\{I^\rho_{o-ie_1}\}_{i=1}^{\infty}
\quad\text{and}\quad 
\{J^\rho_{o-je_2}\}_{j=1}^{\infty}  
\end{align}
be mutually independent sequences of i.i.d. random variables with marginal distributions $I^\rho_{o-ie_1}\sim\text{Exp}(1-\rho)$ and $J^\rho_{o-je_2}\sim\text{Exp}(\rho)$. The boundary variables  in \eqref{IJ} and those   in \eqref{IJ hat} are taken independent of each other. Put  $\rim{G}^\rho_{o,\,o}=0$ and on the boundaries 
\be\label{Gr11} \rim{G}^\rho_{o,\,o-ke_1}=\sum_{i=1}^k I_{o-ie_1} 
\quad\text{and}\quad
\rim{G}^\rho_{o,\,o-le_2}= \sum_{j=1}^l  J_{o-je_2} .   \ee 
Then in the bulk 
for $x=(x_1,x_2)\in o+ \Z_{<0}^2$, 
\be\label{Gr12}
\rim{G}^\rho_{o,\,x}= \max_{1\le k\le o_1-x_1} \;  \Bigl\{  \;\sum_{i=1}^k I_{o-ie_1}  + \rim{G}_{ \,o-ke_1-e_2,x} \Bigr\}  
\bigvee
\max_{1\le \ell\le o_2-x_2}\; \Bigl\{  \;\sum_{j=1}^\ell  J_{o-je_2}  + \rim{G}_{ \,o-\ell e_2-e_1,x} \Bigr\} .
\ee
For a southwest endpoint  $p\in o+\Z_{<0}^2$, let $\rim{Z}^\rho_{o,p}$  be the signed exit point of the geodesic $\pi^{o,p}_{\bbullet}$ of $\rim{G}^\rho_{o,p}$ from the north and east boundaries  of $o+\Z_{<0}^2$.   Precisely,
\begin{align}\label{exit}
\rim{Z}^\rho_{o,\,x}=
\begin{cases}
\argmax{k} \bigl\{ \,\sum_{i=1}^k I_{o-ie_1}  + \rim{G}_{\,o-ke_1-e_2,x} \bigr\},  &\text{if } \pi^{o,x}_1=o-e_1,\\
-\argmax{\ell}\bigl\{  \;\sum_{j=1}^\ell  J_{o-je_2}  + \rim{G}_{\,o-\ell e_2-e_1,x} \bigr\},  &\text{if } \pi_1^{o,x}=o-e_2.
\end{cases}
\end{align}
Similar to \eqref{recu2},  one can backtrack the geodesic $\pi^{o,p}$ in the box $[p,o-e_1-e_2]$ in the following way; for each edge $(x,y)$ (where $y\leq x$) in $[p,o-e_1-e_2]$ assign the weight 
\begin{align}\label{W}
	\rim{\wg}_{y,x}=\rim{G}^\rho_{o,y}-\rim{G}^\rho_{o,x}.
\end{align}
 Let $p_i=\pi^{o,p}_i$, we have
	\begin{align}\label{Gdr}
	&p_m=p,\\
	&p_{i}=
	\begin{cases}
	p_{i+1}+e_1 &\mbox{if } \rim{\wg}_{p_{i+1},p_{i+1}+e_1}<\rim{\wg}_{p_{i+1},p_{i+1}+e_2} \\ \nonumber
	p_{i+1}+e_2 & \mbox{if } \rim{\wg}_{p_{i+1},p_{i+1}+e_1}>\rim{\wg}_{p_{i+1},p_{i+1}+e_2} 
	\end{cases}\quad |\rim{Z}^\rho_{o,p}| \leq i\leq m-1.
	\end{align}
	Since 
	\begin{align}
		\omega_x=(\rim{G}^\rho_{o,x}-\rim{G}^\rho_{o,x+e_1})\wedge (\rim{G}^\rho_{o,x}-\rim{G}^\rho_{o,x+e_2})
	\end{align}
we see that \eqref{Gdr} can be written as
\begin{align}\label{recu Gd}
\pi^{o,x}_n&=x\\
	\omega_{\pi^{o,x}_i}&=w_{\pi^{o,x}_i,\pi^{o,x}_{i-1}}.\nonumber
\end{align}
	The following is a construction we shall refer to often. For general weights $\{Y_x\}_{x\in \Z^2}$ on the lattice and a point $u\in \Z^2$, let $G_{u,x}$ be the LPP by 
	\be
	\Gpp_{u,x}=\max_{x_{\brbullet}\,\in\,\Pi_{u,x}}\sum_{k=0}^{\abs{x-u}_1}Y_{x_k}\quad\text{ for }  y\in u+\Z_{\geq 0}^2. 
	\ee
	Now let $v\in \Z^2$ be such that $u\leq v$. One can construct a new LPP on $\Z^2_{> v}$ as follows. Define the south-west boundary weight
		\begin{align}\label{be4}
	I^{[u]}_{v+ke_1}=& G_{u,v+ke_1}-G_{u,v+(k-1)e_1}
	\qquad\text{for $1 \leq k \leq \infty$},\\
	J^{[u]}_{v+ke_2}=& G_{u,v+ke_2}-G_{u,v+(k-1)e_2}
	\qquad\text{for $1 \leq k \leq \infty$}\nonumber.
	\end{align}
	Let $\{G^{[u]}_{v,x}\}_{x\in \Z_{> v}}$ be the LPP defined through relations \eqref{Gr1}--\eqref{Gr2} using the boundary conditions \eqref{be4} and the bluk weights $\{Y_x\}_{x\in \Z_{> v}}$. We call $G^{[u]}$ \textit{the induced LPP at $v$ by $G_{u,x}$}. 
	The superscript $[u]$ indicates  that   $G^{[u]}$ uses boundary weights determined by the process $G_{u,\bbullet}$  with base point $u$.     Figure \ref{fig-app1}    illustrates the next lemma.  The proof of the lemma is elementary. 
	
	\begin{figure}
		\begin{center}
			\begin{picture}(200,140)(20,-10)
			\put(40,0){\line(1,0){170}} 
			
			\put(90,40){\line(1,0){120}}\put(210,0){\line(0,1){110}}
			\put(40,110){\line(1,0){170}}
			\put(40,0){\line(0,1){110}} \put(90,40){\line(0,1){70}}
			
			\put(37,-3){\Large$\bullet$} 
			\put(30,-2){\small$u$}
			
			\put(86.5,36.5){\Large$\bullet$} 
			\put(80,30){\small$v$}
			
			\put(126.5,36.5){\Large$\bullet$} 
			\put(134,32){\small$x$}
			
			%
			
			\put(206,106){\Large$\bullet$} 
			\put(215,107){\small$y$}
			
			
			\linethickness{3pt} 
			\put(44.5,0){\line(1,0){72}}  \put(115,0){\line(0,1){21}} \put(114,20){\line(1,0){17.5}}
			\put(130,20){\line(0,1){16}}  \put(130,44){\line(0,1){16}} 
			\put(128.5,61.5){\line(1,0){60}} \put(189,60){\line(0,1){21.5}}
			\put(189,80){\line(1,0){22}}  \put(210,78.5){\line(0,1){27}}
			\multiput(94.5,40)(8.5,0){4}{\line(1,0){5.5}}
			
			\end{picture}
		\end{center}  
		\caption{ \small Illustration of Lemma \ref{app-lm1}. Path $u$-$x$-$y$ is a geodesic  of  $G_{u,y}$ and path $v$-$x$-$y$ is a geodesic of  $G^{[u]}_{v,y}$. } \label{fig-app1}
	\end{figure}

	\begin{lemma} \label{app-lm1}  Let $u\le v\le y$ in $\Z^2$.  Then  
		$G_{u,y}=G_{u,v}+G^{[u]}_{v,y}$.    The restriction of any geodesic of $G_{u,y}$ to $v+\Z_{\ge0}^2$ is part of a geodesic of  $G^{[u]}_{v,y}$.    The edges with one endpoint in $v+\Z_{>0}^2$ that belong to 
		a geodesic of $G^{[u]}_{v, y}$  extend to a geodesic of $G_{u,y}$.  
	\end{lemma}  
	In case the process inherited is associated to a stationary process $G^\rho$ we shall use the notation $G^{\rho,[u]}$ to indicate the density $\rho$ as well. Similarly, if $\rim{G}_{u,x}$ is a LLP on $\Z^2_{<u}$ for some $u\in\Z^2$, if $v<u$, we can construct the induced process $\rim{G}^{[u]}_{v,x}$ on $\Z^2_{<v}$. A result similar to Lemma \ref{app-lm1} holds for $\rim{G}_{u,x}$ and $\rim{G}^{[u]}_{v,x}$.

\section{Busemann functions} 
\subsection{Existence and properties of Busemann functions}
Let $(\Omega,\cF,\P)$ be a probability space and let $\{\tau_z\}_{z\in\Z^2}$ be a group of translations on $\Omega$.
\begin{definition}
A measurable function $B:\Omega \times\Z^2 \times \Z^2\rightarrow \R$ is a stationary cocycle if it satisfies these two conditions for $\P$-a.e. $\omega$ and all $x,y,z\in \Z^2$:
\begin{align}
	B(\omega,x+z,y+z)&=B(\tau_z\omega,x,y) \quad \text{(stationarity)}\\
	B(\omega,x,y)+B(\omega,y,z)&=B(\omega,x,z)\quad \text{(additivity)}
\end{align}
\end{definition}
Given a down-right path $\cY\in \dr$, the  lattice decomposes  into a disjoint union  $\Z^2=\cG_-\cup\cY\cup\cG_+$  where   the two regions are 
\[  \cG_-=\{x\in \Z^2:  \exists j\in\Z_{>0}  \text{ such that  } x+j(e_1+e_2)\in \cY\} \]
and 
\[  \cG_+=\{x\in \Z^2:  \exists j\in\Z_{>0}  \text{ such that } x-j(e_1+e_2)\in \cY\} . \]
\\

\begin{definition} \label{v:d-exp-a} 
	Let $0<\alpha<1$.   Let us say that a process 
	\be\label{IJw800} 
	\{ \eta_{x} ,\,   I_{x}, \, J_{x},  \,\wc\eta_{x}  :  x\in \Z^2\} 
	\ee
	is an {\it exponential-$\alpha$  last-passage percolation system}  if the following properties  {\rm (a)--(b)} hold. \\[-7pt] 
	\begin{enumerate}[{\rm(a)}]\itemsep=7pt 
		\item  The process is stationary with marginal distributions  
		\be\label{w19}\begin{aligned}
			\eta_{x},\, \wc\eta_x \sim \text{\rm Exp}(1),  \quad I_{x}\sim \text{\rm Exp}(1-\alpha)  , \quad \text{ and }  \quad  J_{x}\sim \text{\rm Exp}(\alpha)  .  
		\end{aligned}\ee 
		For any down-right path $\cY=(y_k)_{k\in\Z}$  in   $\Z^2$, the random variables 
		\be\label{v:Y4}   
		\{\wc\eta_z:  z\in\cG_-\}, \quad  \{  \tincr(\{y_{k-1},y_k\}): k\in\Z\}, \quad\text{and}\quad 
		\{  \eta_x: x\in\cG_+\} 
		\ee
		are all mutually independent, where the undirected edge variables $\tincr(e)$ are defined  as 
		\be\label{v:tincr6}   \tincr(e)= \begin{cases}  I_x &\text{if $e=\{x-e_1,x\}$}\\ J_x &\text{if $e=\{x-e_2,x\}$.} \end{cases}  \ee

		\item  The following equations  are in force at all $x\in\Z^2$: 
		\begin{align}
		\label{IJw5.1.7}  \wc\eta_{x-e_1-e_2}&=I_{x-e_2}\wedge J_{x-e_1} 
		\\
		\label{IJw5.2.7}  I_x&=\eta_x+(I_{x-e_2}-J_{x-e_1})^+ 
		\\
		\label{IJw5.3.7} J_x&=\eta_x+(I_{x-e_2}-J_{x-e_1})^-. 
		\end{align}

	\end{enumerate}
\end{definition}

The following Theorem was proven in \cite{sepp-arxiv-18} 

\begin{theorem}\label{t:buse}  
	For each  $0<\alpha<1$  there exist a stationary   cocycle $B^\alpha$ and a family of random weights  $\{\Xw^\alpha_x\}_{x\in\Z^2}$  on $\OSP$   with the following properties.  
	\begin{enumerate}\itemsep=3pt
		\item[{\rm(i)}]   For each $0<\alpha<1$,  process 
		\begin{align}\label{v:bf}
			\{  \Xw^\alpha_x, \, B^\alpha_{x-e_1,x}, \, B^\alpha_{x-e_2,x}, \,\Yw_x:  x\in\Z^2\}
		\end{align}
		is an exponential-$\alpha$ last-passage system as described in Definition \ref{v:d-exp-a}. 
		
		\item[{\rm(ii)}]     There exists a single event $\Omega_2$ of full probability such that for all $\w\in\Omega_2$,   all $x\in\Z^2$ and all $\lambda<\rho$ in $(0,1)$ we have the inequalities 
		\be\label{v:853.1} 
		B^\lambda_{x,x+e_1}(\w) \le  B^\rho_{x,x+e_1}(\w) \quad\text{and}\quad 
		B^\lambda_{x,x+e_2}(\w) \ge  B^\rho_{x,x+e_2}(\w).   
		\ee
		Furthermore,  for all $\w\in\Omega_2$ and $x,y\in\Z^2$,  the function $\lambda\mapsto B^\lambda_{x,y}(\w)$ is right-continuous with left limits. 
		
		\item[{\rm(iii)}]     For each fixed $0<\alpha<1$ there exists an event $\Omega^{(\alpha)}_2$ of full probability such that the following holds: for each $\w\in\Omega^{(\alpha)}_2$  and any  sequence $v_n\in\Z^2$ such that $\abs{v_n}_1\to\infty$ and 
		\be\label{v:853.3}     \lim_{n\to\infty}  \frac{v_n}{\abs{v_n}_1} \, = \,\xi(\alpha)\,=\, 
		\biggl(  \frac{(1-\alpha)^2}{(1-\alpha)^2+\alpha^2} \,,   \frac{\alpha^2}{(1-\alpha)^2+\alpha^2}\biggr) ,  
		\ee
		we have the limits  
		\be\label{v:855}   
		B^\alpha_{x,y}(\w)  =\lim_{n\to\infty} [ G_{x, v_n}(\w)-G_{y,v_n}(\w)]  \qquad\forall x,y\in\Z^2. \ee
		The LPP process $G_{x,y}$ is now defined by \eqref{v:G}.  
		Furthermore,  for all $\w\in\Omega^{(\alpha)}_2$ and $x,y\in\Z^2$,  
		\be\label{v:855.7}   \lim_{\lambda\to\alpha}B^\lambda_{x,y}(\w)=B^\alpha_{x,y}(\w). \ee 
		
	\end{enumerate} 
	
\end{theorem} 

	{\bf Busemann Functions and Infinite Geodesics.} Fix $x\in \Z^2$. An infinite up-right path $\pi^{x,\infty}_\bbullet$ originating at $x$ is called a geodesic if for all $m,n\in\Z_{>0}$ such that $m<n$, the path  $\{\pi^{x,\infty}_l\}_{l\in\lzb m,n\rzb}$ is a geodesic. We say a geodesic has  direction $\xi\in \ri \Uset$ if
	\begin{align}
	\lim_{n\rightarrow \infty} \frac{\pi^{x,\infty}_n(2)}{\pi^{x,\infty}_n(1)}=\frac{\xi_2}{\xi_1}.
	\end{align}
	To each direction $\xi\in \ri \Uset$  we associate a density $\rho\in(0,1)$ through the relations
	\begin{align}\label{rhoxi}
	\xi(\rho)=\left(\frac{(1-\rho)^2}{(1-\rho)^2+\rho^2},\frac{\rho^2}{(1-\rho)^2+\rho^2}\right)
	\end{align}
	In some literature $\xi(\rho)$ is called  the  {\it  characteristic direction}   associated with the parameter $\rho$. It is known that for $\xi\in \ri \Uset$, with probability one, every point $x\in \Z^2$ has a unique geodesic $\pi^{x,\infty}$ of direction $\xi$.
	Busemann functions can be used to construct infinite geodesic. Consider the family of random variables
	\begin{align}
	\{  B^\alpha_{x-e_1,x}, \, B^\alpha_{x-e_2,x}, \,\Yw_x:  x\in\Z^2\}
	\end{align}
	defined in \eqref{v:bf}. Let $\xi:=\xi(\alpha)$ be the characteristic direction associated with $\rho$. Let us denote by $\pi^{x,\xi\infty}$ the  infinite geodesic with respect to the weights $\{\omega_x\}_{x\in\Z}$, starting from $x$ in direction $\xi$. In \cite{geor-rass-sepp-17-buse}, it was shown that one can trace the infinite geodesic $\pi^{x,\infty}$ by following the gradient of the Busemann function $B^\alpha$. Let $\{p_i\}_{i\in \Z_{>0}}$  be an enumeration of the vertices in $\pi^{x,\xi\infty}$, i.e.
	\begin{align*}
		p_i=\pi^{x,\xi\infty}_i \quad i\in \Z_{\geq 0},
	\end{align*}
	where $p_0=x$. Then the vertices $\{p_i\}_{i\in\Z_{\geq 0}}$ are given recursively through
	\begin{align}\label{recu3}
	&p_0=x,\\
	&p_i=
	\begin{cases}
	p_{i-1}+e_1 &\mbox{if } B^\alpha_{p_{i-1},p_{i-1}+e_1}<B^\alpha_{p_{i-1},p_{i-1}+e_2} \\ \nonumber
	p_{i-1}+e_2 & \mbox{if } B^\alpha_{p_{i-1},p_{i-1}+e_1}>B^\alpha_{p_{i-1},p_{i-1}+e_2}\end{cases}\quad i\in \Z_{> 0}.
	\end{align}
	Note that in \eqref{recu3} $p_i$ is attained by taking an up\textbackslash right step from the point $p_{i-1}$ in the direction where the minimal increment of the Busemann function is attained. Since $\omega_x=B^\alpha_{x,x+e_1}\wedge B^\alpha_{x,x+e_2}$, $\pi^{x,\xi\infty}$ is the unique path that satisfies 
	\begin{align}
	\pi^{x,\infty}_0&=x\\
	\Yw_{\pi^{x,\infty}_i}&=B^\rho_{\pi^{x,\infty}_i,\pi^{x,\infty}_{i+1}}.
	\end{align}
	\begin{lemma}\label{lem:ing}
		Let $x\in \Z^2$ and $\xi_1,\xi_2\in \ri\Uset$ such that $\xi_1\og \xi_2$. For $i\in\{1,2\}$ let $\pi^{x,\xi_i\infty}$ be the infinite geodesic starting from $x$ in direction $\xi_i$. Then 
		\begin{align}\label{gin}
			\pi^{x,\xi_1\infty}\og \pi^{x,\xi_2\infty}
		\end{align}
	\end{lemma}
	\begin{proof}
	Suppose \eqref{gin} does not hold. Then there exists $y_1\in \pi^{x,\xi_1\infty}$ and $y_2\in \pi^{x,\xi_2\infty}$ such that $x \leq y_1,y_2$ and $y_2 \ogs y_1$. Since both geodesics have a direction and $\xi_1\og \xi_2$, there exists $w_1\in \pi^{x,\xi_1\infty}$ and $w_2\in \pi^{x,\xi_2\infty}$ such that $y_1,y_2\leq w_1,w_2$ and $w_1 \og w_2$. It follows that there exists a point $x \leq z$ such that $z\in \pi^{x,\xi_1\infty}\cap \pi^{x,\xi_2\infty}$ and that the geodesics
	\begin{align}
		\gamma_1&=\pi^{x,\xi_1\infty}\cap [x,z]\\
		\gamma_2&=\pi^{x,\xi_2\infty}\cap [x,z],
	\end{align}
	start from $x$ and terminate at $z$. As $y_2 \ogs y_1$ it follows that $\gamma_1\neq \gamma_2$ which violates the uniqueness of geodesics.
	\end{proof}
	\subsection{Coupling Busemann Functions}
	In \cite{fan-sepp-arxiv} a coupling between  Busemann functions of different densities was given which relies on the queueing mapping. Consider the queueing mapping $D:\R_+^\Z \rightarrow \R_+^\Z$ from Appendix \ref{app:queues}. 
	\begin{lemma}\label{Bus cop}
	  Let $0<\rhodown{\rho}<\rhoup{\rho}<1$. There exists a coupling of $B^{\rhodown{\rho}}$ and $B^{\rhoup{\rho}}$ such that
		\begin{enumerate}[{\rm(i)}] 
		\item for every $x\in\Z^2$ 
		\begin{align}
			B^{\rhodown{\rho}}_{x,x+e_2} &\geq B^{\rhoup{\rho}}_{x,x+e_2}\\
			B^{\rhoup{\rho}}_{x,x+e_1} &\geq B^{\rhodown{\rho}}_{x,x+e_1}\nonumber.	
		\end{align}
		\item for every $x\in\Z^2$
		\begin{align*}
			\Big(B^{\rhoup{\rho}}_{x+ie_1,x+(i+1)e_1},B^{\rhodown{\rho}}_{x+ie_1,x+(i+1)e_1}\Big)_{i\in \Z} &\sim \nu^{1-\rhoup{\rho},1-\rhodown{\rho}}\\
			\Big(B^{\rhodown{\rho}}_{x+ie_2,x+(i+1)e_2},B^{\rhoup{\rho}}_{x+ie_2,x+(i+1)e_2}\Big)_{i\in \Z} &\sim \nu^{\rhodown{\rho},\rhoup{\rho}},
		\end{align*}
		where for $0<\lambda<\rho<1$ $\nu^{\lambda,\rho}$ is the distribution defined in \eqref{qd}. 
		\item fix  $x\in\Z^2_{> 0}$. For every $k,l\in \Z$  the following sets of random variables are independent 
		\begin{align}
			&\{B^{\rhoup{\rho}}_{x+ie_2,x+(i+1)e_2}\}_{0 \leq i \leq l-1 },\quad \{B^{\rhodown{\rho}}_{x+ie_2,x+(i+1)e_2}\}_{-k \leq i \leq -1 } \nonumber\\
			\text{and so are}\\
			&\{B^{\rhodown{\rho}}_{x+ie_1,x+(i+1)e_1}\}_{0 \leq i \leq l-1 },\quad \{B^{\rhoup{\rho}}_{x+ie_1,x+(i+1)e_1}\}_{-k \leq i \leq -1 } \nonumber.
		\end{align}
	\end{enumerate}  
	\end{lemma}

\section{stabilization}\label{sec:stb}

In this section we prove Theorem \ref{thm:loc} and Theorem \ref{thm:stb}. Let us define the event
\begin{align*}
	\mathcal{H}^{\xi,M}=\{\rim{G}_{N\xi,y}-\rim{G}_{N\xi,x}=B^{\rho(\xi)}_{x,y}\text{ for all }(x,y)\in  \mathcal{E}(\Rb^{\xi,M}) \}.
\end{align*}
 $\mathcal{H}^{\xi,M}$ is the event where the increments of $G$ along the edges in $\mathcal{E}(\Rb^{\xi,M})$ coincide with those of the Busemann function associated with the direction $\xi$. It will be clear from the proof that $\mathcal{H}^{\xi,M}$ and $\st^{\xi,M}$ defined in \eqref{Sf} and are equivalent events.
\subsection{Bounds on $\P(\st^{\xi,M})$ and $\P(\mathcal{H}^{\xi,M})$}
Let $\xi\in \ri\Uset$ and let  $\rhodown{\rho}(\xi)=\rho(\xi)-rN^{-\frac{1}{3}}$ and $\rhoup{\rho}(\xi)=\rho(\xi)+rN^{-\frac{1}{3}}$. We also let $\rim{o}=\xi N+e_1+e_2$. Assign weights on the edges of the north-east boundary of $\Rb^{\xi,N}$ by
\begin{align}\label{BW1}
&I^\rhoup{\rho}_i=B^\rhoup{\rho}_{\rim{o}_N-(i+1)e_1,\rim{o}_N-ie_1} \quad 0\leq i \leq N\xi_1 \\ &J^\rhoup{\rho}_i=B^\rhoup{\rho}_{\rim{o}_N-(i+1)e_2,\rim{o}_N-ie_2} \quad 0\leq i \leq N\xi_2. \nonumber
\end{align} 
Use the boundary weights $\{I^\rhoup{\rho}_i\}_{0\leq i \leq N\xi_1},\{J^\rhoup{\rho}_i\}_{0\leq i \leq N\xi_2}$  and the bulk weights $\{\omega_x\}_{x\in \Rb^{\xi,N}}$ to construct the stationary LPP $\rim{G}^\rhoup{\rho}$  as in \eqref{Gr11}--\eqref{Gr12}. Similarly we construct $\rim{G}^\rhodown{\rho}$. As in \eqref{exit} we let $\rim{Z}^\rhoup{\rho}_{\rim{o},x}$($\rim{Z}^\rhodown{\rho}_{\rim{o},x}$) denote the exit point of the  geodesic $\pi^{\rim{o},x}$ of $\rim{G}^\rhoup{\rho}$($\rim{G}^\rhodown{\rho}$). Let
\begin{align}\label{Adef}
	\rim{A}^{\xi,M}=\Big\{\sup_{x\in \Rb^{\xi,M}}\rim{Z}^{\rhodown{\rho}(\xi)}_{o,x}<0\Big\}\bigcap \Big\{\inf_{x\in \Rb^{\xi,M}}\rim{Z}^{\rhoup{\rho}(\xi)}_{o,x}>0\Big\}
\end{align}
$\rim{A}^{\xi,M}$ is the event that for any $x\in\Rb^{\xi,M}$, the geodesic $\pi^{\rim{o},x}$ of  $\rim{G}^\rhoup{\rho}$($\rim{G}^\rhodown{\rho}$) crosses the boundary of $\Rb^{\xi,N}$(not to be confused with $\Rb^{\xi,M}$!) from the north (east) boundary. Recall the definition of $\xi(\rho)$ in \eqref{v:853.3}. Define
\begin{align}\label{xibar}
	\rhoup{\xi}=\xi(\rhoup{\rho}) \quad \text{and }\quad \rhodown{\xi}=\xi(\rhodown{\rho}).
\end{align}
\begin{lemma}\label{og2}
	Let $\xi\in \ri\Uset$ and $M>0$. On the set $\rim{A}^{\xi,M}$
	\begin{align}\label{gin1}
		\pi^{x,\rhoup{\xi}\infty} \og \pi^{x,\xi N} \og \pi^{x,\rhodown{\xi}\infty} \quad x\in \Rb^{\xi,M}.
	\end{align}
\end{lemma}
\begin{proof}
	We show the first inequality in \eqref{gin1}, the second one is analogous. Note that by our construction, the exit point $Z^{\rhoup{\rho}}_{\rim{o},x}$ is the leftmost point in the north boundary of $\Rb^{\xi,N+1}$ that belongs to $\pi^{x,\rhoup{\xi}\infty}$, more precisely, on the set $\rim{A}^{\xi,M}$, for any $x\in \Rb^{\xi,M}$
	\begin{align}
		\inf \{i:\rim{o}_N-ie_1 \in \pi^{x,\rhoup{\xi}\infty}\} = |Z^{\rhoup{\rho}}_{\rim{o},x}|>0.
	\end{align}
	By order of geodesics
	\begin{align}
	\sup \{i:\rim{o}_N-e_2-ie_1 \in \pi^{x,\rhoup{\xi}\infty}\} \geq \inf \{i:\rim{o}_N -e_2 -ie_1 \in \pi^{x,\xi N}\}, 
	\end{align}
	which in turn implies that 
	\begin{align}
	\pi^{x,\rhoup{\xi}\infty} \og \pi^{x,\xi N}.
	\end{align}
\end{proof}
Let $\rim{o}_M=M\xi+e_1+e_2$ be the upper right corner of $\Rb^{\xi,M}$. Assign weights on the edges of the north-east boundary of $\Rb^{\xi,M}$ by
\begin{align}\label{BW}
	&I^\rhoup{\rho}_i=B^\rhoup{\rho}_{\rim{o}-(i+1)e_1,\rim{o}-ie_1} \quad 0\leq i \leq M\xi_1 \\ &J^\rhoup{\rho}_i=B^\rhoup{\rho}_{\rim{o}-(i+1)e_2,\rim{o}-ie_2} \quad 0\leq i \leq M\xi_2. \nonumber
\end{align}
Similarly we define $\{I^\rhodown{\rho}_i\}_{0\leq i \leq M\xi_1}$ and $\{J^\rhodown{\rho}_i\}_{0\leq i \leq M\xi_2}$. Define the event
\begin{align}\label{Cd}
	C^{\xi,M}=\Big\{I^\rhoup{\rho}_i=I^\rhodown{\rho}_i\Big\}_{0\leq i \leq M\xi_1} \bigcap \Big\{J^\rhoup{\rho}_i=J^\rhodown{\rho}_i\Big\}_{0\leq i \leq M\xi_2}. 
\end{align}
\medskip
\begin{lemma}\label{lem:og}
	On the event $C^{\xi,M}$
	\begin{align}
		&B^\rhoup{\rho}_e=B^{\rho}_e=B^\rhodown{\rho}_e \quad e\in \mathcal{E}(\Rb^{\xi,M})\label{ge3}\\
		&\pr^{\xi,M}(\pi^{x,\rhoup{\xi}\infty}) = \pr^{\xi,M}(\pi^{x,\rhodown{\xi}\infty}) \label{ge} 
	\end{align}
\end{lemma}
\begin{proof}
	First we note that for every $0<\rho(\xi)<1$ the mapping 
	\begin{align}
		e\mapsto B^\rho_e \quad e\in \mathcal{E}(\Rb^{\xi,M})
	\end{align}
	is determined uniquely by the bulk weights $\{\omega\}_{x\in \Rb^{\xi,M}}$ and the boundary weights $\{I^\rho_i\}_{0\leq i \leq M\xi_1}$ and $\{J^\rho_i\}_{0\leq i \leq M\xi_2}$ constructed as in \eqref{BW}. This follows from the recursive relation (\eqref{IJw5.2.7}--\eqref{IJw5.3.7})
	\begin{align*}
		B^\rho_{x,x+e_1} &= \w_x + ( B^\rho_{x+e_2, x+e_1+e_2} - B^\rho_{x+e_1, x+e_1+e_2} )^+\\
		B^\rho_{x,x+e_2} &= \w_x + ( B^\rho_{x+e_2, x+e_1+e_2} - B^\rho_{x+e_1, x+e_1+e_2} )^-.
	\end{align*}
	On the event $C^{\xi,M}$ the boundary conditions in \eqref{BW}
	 are equal for the processes $\rim{G}^{\rhoup{\rho}}$ and $\rim{G}^{\rhodown{\rho}}$ i.e.
	\begin{align}\label{be2}
	I^\rhoup{\rho}_{\rim{o}_N-ke_1}=& I^\rhodown{\rho}_{\rim{o}_N-ke_1}
	\qquad\text{for $1 \leq k \leq \xi_1 M$}.\\
	J^\rhoup{\rho}_{\rim{o}_N-ke_2}=& J^\rhodown{\rho}_{\rim{o}_N-ke_2} 
	\qquad\text{for $1 \leq k \leq \xi_2 M$}\nonumber.
	\end{align}
	As both $\{B^\rhoup{\rho}_e\}_{e\in \mathcal{E}(\Rb^{\xi,M})}$ and $\{B^\rhodown{\rho}_e\}_{e\in \mathcal{E}(\Rb^{\xi,M})}$ use the same bulk weights, by \eqref{be2}, we  conclude that on $C^{\xi,M}$
	\begin{align}\label{be}
	B^\rhoup{\rho}_e=B^\rhodown{\rho}_e \quad e\in \mathcal{E}(\Rb^{\xi,M}). 
	\end{align}
	By Lemma \ref{lm:crs} (the Crossing Lemma), for every $e\in \mathcal{E}(\Rb^{\xi,M})$
	\begin{align*}
		B^\rhoup{\rho}_e\leq B^{\rho}_e\leq B^\rhodown{\rho}_e \text{ or }B^\rhodown{\rho}_e\leq B^{\rho}_e\leq B^\rhoup{\rho}_e \quad \forall e\in \mathcal{E}(\Rb^{\xi,M})
	\end{align*}
	which, together with \eqref{be}, implies \eqref{ge3}.
	 The geodesics  $\pr^{\xi,M}(\pi^{x,\rhoup{\xi}\infty})$ and  $\pr^{\xi,M}(\pi^{x,\rhodown{\xi}\infty})$ are determined by $	\{B^\rhoup{\rho}_e\}_{e\in \mathcal{E}(\Rb^{\xi,M})}$ and $\{B^\rhodown{\rho}_e\}_{e\in \mathcal{E}(\Rb^{\xi,M})}$ respectively using \eqref{recu3}. \eqref{ge} is now implied by \eqref{ge3}.
\end{proof}
\medskip
\begin{corollary}\label{cor:ge}
	On $C^{\xi,M}\cap \rim{A}^{\xi,M}$
	\begin{align}
	\{\rim{G}_{N\xi,y}-\rim{G}_{N\xi,x}=B^{\rho(\xi)}_{x,y}&\text{ for all }(x,y)\in  \mathcal{E}(\Rb^{\xi,M}) \}\label{ge4}.\\
	\pr^{\xi,M}(\pi^{x,\rhoup{\xi}\infty}) = \pr^{\xi,M}(\pi^{x,{\xi}\infty})&=\pr^{\xi,M}(\pi^{x,{\xi}N})= \pr^{\xi,M}(\pi^{x,\rhodown{\xi}\infty}) \quad x\in \Rb^{\xi,M}\label{ge2}. 
	\end{align}
\end{corollary}
\begin{proof}
	By Lemma \ref{lm:crs}, on the event $\rim{A}^{\xi,M}$ 
	\begin{align*}
			B^\rhoup{\rho}_{x,x+e_2}\leq \rim{G}_{N\xi,x+e_2}&-\rim{G}_{N\xi,x}\leq B^\rhodown{\rho}_{x,x+e_2}\quad \forall (x,x+e_2)\in \mathcal{E}(\Rb^{\xi,M}) \\&\text{ and }\\
			B^\rhodown{\rho}_{x,x+e_1}\leq \rim{G}_{N\xi,x+e_1}&-\rim{G}_{N\xi,x}\leq B^\rhoup{\rho}_{x,x+e_1} \quad \forall (x,x+e_1)\in \mathcal{E}(\Rb^{\xi,M}),
	\end{align*}
	which implies \eqref{ge4}, using \eqref{ge3}.
	By Lemma \ref{lem:ing} we see that 
	\begin{align}\label{og3}
		\pi^{x,\rhoup{\xi}\infty}\og \pi^{x,{\xi}\infty} \og \pi^{x,\rhodown{\xi}\infty} \quad x\in \Z^2.
	\end{align}
	By Lemma \ref{og2}, on $\rim{A}^{\xi,M}$
	\begin{align}\label{og4}
	\pi^{x,\rhoup{\xi}\infty} \og \pi^{x,\xi N} \og \pi^{x,\rhodown{\xi}\infty} \quad x\in \Rb^{\xi,M}.
	\end{align}
	Lemma \ref{lem:og} along with \eqref{og3} and \eqref{og4} imply the result.
\end{proof}
For $M>0$, recall the set $\st^{\xi,M}$ in \eqref{Sf}. Using \ref{cor:ge} we have the following.
\medskip
It should be now clear to the reader that $\mathcal{H}^{\xi,M}=\st^{\xi,M}$. We have the following control on the probability of these events.
\begin{corollary}
\begin{align}
	\P(\mathcal{H}^{\xi,M})=\P(\st^{\xi,M})\geq \P\big(C^{\xi,M}\cap \rim{A}^{\xi,M}\big)\geq 1-\P((C^{\xi,M})^\mathrm{c})-\P((\rim{A}^{\xi,M})^\mathrm{c}) \label{ubs}.
\end{align}
\end{corollary}
\subsection{Upper bound on $\P((\rim{A}^{\xi,M})^c)$}
\medskip 
\begin{lemma}\label{lem-lb1} Fix $\xi\in (0,1)$ and $M>0$ such that $M\leq ct^\frac{2}{3}$ for some constant $c>0$. Let $o=\xi N$. There exist constants $C(c,\xi),C_1,N_0(\xi,c)>0$, locally bounded in $c$, such that for all $N>N_0$ and $1\leq r \leq N^\frac{1}{3}(\log(N))^{-1}$ and all $x\in \Rb^{\xi,M}$
	\be\label{lb-1}   \P\Big(\sup_{x\in \Rb^{\xi,M}}\rim{Z}^{\rhodown{\rho}(\xi)}_{o,x}>0\Big) \le \frac{C_1}{r^3}
	\ee
	and 
	\be\label{lb-2}   \P\Big(\inf_{x\in \Rb^{\xi,M}}\rim{Z}^{\rhoup{\rho}(\xi)}_{o,x}< 0\Big) \le \frac{C_1}{r^3},
	\ee
	where  $\rhodown{\rho}(\xi)=\rho(\xi)-rN^{-\frac{1}{3}}$ and $\rhoup{\rho}(\xi)=\rho(\xi)+rN^{-\frac{1}{3}}$.
\end{lemma}

\begin{proof}
	We only prove \eqref{lb-2} as \eqref{lb-1} is similar.  Given  $\xi\in\ri\Uset$, abbreviate $\rho=\rho(\xi)$ and  $\rhoup{\rho}=\rhoup{\rho}(\xi)$. 
	Let $x^0=(M\xi_1,0)$ be the lower-right corner of $\Rb^{\xi,M}$.
	By the order on geodesics we have
	\begin{align*}
	\Big\{\inf_{x\in \Rb^{\xi,M}}\rim{Z}^{\rhoup{\rho}}_{o,\,x}<0\Big\}\subset \Big\{\rim{Z}^{\rhoup{\rho}}_{o,\,x_0}<0\Big\},
	\end{align*}
	which implies
	\begin{align}\label{epl}
	\P\Big(\inf_{x\in \Rb^{\xi,M}}\rim{Z}^{\rhoup{\rho}}_{o,\,x}<0\Big)\leq \P\Big(\rim{Z}^{\rhoup{\rho}}_{o,\,x_0}<0\Big).
	\end{align}
	In order to upper bound \eqref{epl} we must show that the characteristic line of direction $\xi(\rhoup{\rho})$ that leaves from  $o$ goes, on the scale of $N^\frac{2}{3}$, well below the point $x^0$. We have, via \eqref{rhoxi},
	\begin{align}\label{comp}
	&N\xi_2-\frac{(N\xi_1-M\xi_1)\rhoup{\rho}^2}{(1-\rhoup{\rho})^2}=\frac{N\xi_2(1-\rhoup{\rho})^2-(N\xi_1-M\xi_1)\rhoup{\rho}^2}{(1-\rhoup{\rho})^2}\\&=\frac{\xi_2[-2(1-\rho)rN^\frac{2}{3}+r^2N^\frac{1}{3}]-\xi_1[2\rho rN^\frac{2}{3}+r^2N^\frac{1}{3}]+\xi_1M\rhoup{\rho}^2}{(1-\rhoup{\rho})^2}\nonumber\\
	&=-rN^\frac{2}{3}\frac{[\xi_22(1-\rho)+\xi_12\rho ]}{(1-\rhoup{\rho})^2}+\frac{M\xi_1\rhoup{\rho}^2}{(1-\rhoup{\rho})^2}+\frac{(\xi_2-\xi_1)r^2N^\frac{1}{3}}{(1-\rhoup{\rho})^2}\nonumber\\
	&\leq-rN^\frac{2}{3}\frac{[\xi_22(1-\rho)+\xi_12\rho ]}{(1-\rhoup{\rho})^2}+\frac{cN^{\frac{2}{3}}\xi_1\rhoup{\rho}^2}{(1-\rhoup{\rho})^2}+\frac{(\xi_2-\xi_1)r^2N^\frac{1}{3}}{(1-\rhoup{\rho})^2}\nonumber\\
	&\leq -rN^\frac{2}{3}\frac{[\xi_22(1-\rho)+\xi_12\rho -r^{-1}c\xi_1(\rho^2-2\rho rN^{-\frac{1}{3}}+r^2N^{-\frac{2}{3}})-rN^{-\frac{1}{3}}]}{(1-\rhoup{\rho})^2}.\label{comp1}	
	\end{align}
	For large enough $N$ and $r\leq N^\frac{1}{3}(\log(N))^{-1}$
	plug into \eqref{comp1} to obtain
	\begin{align}\label{comp2}
	&N\xi_2-\frac{(N\xi_1-M\xi_1)\rhoup{\rho}^2}{(1-\rhoup{\rho})^2}\\
	&\leq -rN^\frac{2}{3}\frac{[\xi_22(1-\rho)+\xi_12\rho -r^{-1}c\xi_1(\rho^2-2\rho rN^{-\frac{1}{3}}+r^2N^{-\frac{2}{3}})-rN^{-\frac{1}{3}}]}{(1-\rhoup{\rho})^2}\nonumber\\
	&\leq -rN^\frac{2}{3}\frac{[\xi_22(1-\rho)+\xi_12\rho -r^{-1}c\xi_1\rho^2+c\xi_1(2\rho  N^{-\frac{1}{3}}+N^{-\frac{1}{3}}(\log(N))^{-1})-(\log(N))^{-1}]}{(1-\rho)^2},\nonumber
	\end{align}
	such that for $N$ large enough
	\begin{align*}
		N\xi_2-\frac{(N\xi_1-M\xi_1)\rhoup{\rho}^2}{(1-\rhoup{\rho})^2}
		\leq -rN^\frac{2}{3}\frac{[\xi_2(1-\rho)+\xi_1\rho -r^{-1}c\xi_1\rho^2]}{(1-\rho)^2}.
	\end{align*} 
	For
	\begin{align*}
	r>\frac{c\xi_1\rho^2}{[\xi_2(1-\rho)+\xi_1\rho -(1-\rho)^2]}\vee 1,
	\end{align*}
	the right hand side of \eqref{comp2} is smaller than $-N^\frac{2}{3}$. This in turn implies that there exists a constant $C'(\xi,c)>0$ (locally bounded in $c$) such that
	\begin{align*}
	N\xi_2-\frac{(N\xi_1-M\xi_1)\rhoup{\rho}^2}{(1-\rhoup{\rho})^2}\leq -C'(\xi,c)rN^\frac{2}{3},
	\end{align*}
	 It then follows by \cite{sepp-cgm-18}[Corollary 5.10] that there exists a constant $C_1(\xi,c)>0$
	\begin{align*}
	\P\Big(\rim{Z}^{\rhoup{\rho}}_{o,\,x_0}<0\Big)\leq C_1r^{-3},
	\end{align*}
	which proves the result. 	
\end{proof}
\begin{corollary}\label{cor:ubA}
	Fix $\xi\in (0,1)$ and $M>0$ such that $M\leq ct^\frac{2}{3}$ for some constant $c>0$. There exists $C(c,\xi)>0$, locally bounded in $c$, such that
	\begin{align}
		\P((\rim{A}^{\xi,M})^c)\leq \frac{C}{r^3}\label{ubA}.
	\end{align}
\end{corollary}
\begin{proof}
	By the definition of $\rim{A}^{\xi,M}$ we see that 
	\begin{align}\label{se}
		(\rim{A}^{\xi,M})^c\subseteq  \Big\{\sup_{x\in \Rb^{\xi,M}}\rim{Z}^{\rhodown{\rho}(\xi)}_{o,x}>0\Big\}\cup \Big\{\inf_{x\in \Rb^{\xi,M}}\rim{Z}^{\rhoup{\rho}(\xi)}_{o,x}< 0\Big\}
	\end{align}
	Taking probability on both sides of \eqref{se} and using \eqref{lb-1} and \eqref{lb-2} we obtain the result.
\end{proof}
\subsection{Upper bound on $\P((C^{\xi,M})^c)$}
Define 
\begin{align*}
	C^{\xi,M}_1&=\{B^{\rhoup{\xi}}_{\rim{o}_M-ke_1,\rim{o}_M-(k-1)e_1}=B^{\rhodown{\xi}}_{\rim{o}_M-ke_1,\rim{o}_M-(k-1)e_1}\}
	\qquad\text{for $1 \leq k \leq \xi_1 M$}\\
	C^{\xi,M}_2&=\{B^{\rhoup{\xi}}_{\rim{o}_M-ke_2,\rim{o}_M-(k-1)e_2}=B^{\rhodown{\xi}}_{\rim{o}_M-ke_2,\rim{o}_M-(k-1)e_2}\}
	\qquad\text{for $1 \leq k \leq \xi_2 M$}.
\end{align*}
Recall \eqref{Cd} and note that 
\begin{align}\label{C}
	C^{\xi,M}=C^{\xi,M}_1\cap C^{\xi,M}_2.
\end{align}
This subsection is aimed at proving the following.
\begin{proposition}\label{prop:Cub}	Let $\xi\in \ri\Uset$ and $c>0$ and let $M>0$ such that $M\leq 
	cN^\frac{2}{3}$. There exists $C(\xi,c)>0$, locally bounded in $c$, such that
	\begin{align}
		\P\Big(\big(C^{\xi,M}\big)^c\Big)\leq C N^{-\frac{1}{4}}M^\frac{3}{8}\label{ubC}.
	\end{align}
\end{proposition}
Before we prove Proposition \ref{prop:Cub} we obtain some auxiliary results. As was noted in Lemma \ref{Bus cop}[ii], for $\rhoup{\xi}\og \rhodown{\xi}$ (and therefore $\rhodown{\rho} \leq \rhoup{\rho}$)
\begin{align*}
&\P	\Big(B^{\rhoup{\xi}}_{\rim{o}_M-ke_1,\rim{o}_M-(k-1)e_1}=B^{\rhodown{\xi}}_{\rim{o}_M-ke_1,\rim{o}_M-(k-1)e_1} \qquad \text{for $1 \leq k \leq \xi_1 M$} \Big)\\[3pt]
&=\nu^{1-\rhoup{\rho},1-\rhodown{\rho}}(d_i=s_i \qquad \text{for $1 \leq i \leq \xi_1 M$}),
\end{align*}
where $\depav=D(\arrv,\servv)$(see Appendix \ref{app:queues}), and  $\arrv=(\arr_j)_{j\in\Z}$ and  $\servv=(\serv_j)_{j\in\Z}$ are two independent i.i.d sequences of exponential random variables of intensity $\rhodown{\rho}$ and $\rhoup{\rho}$ respectively, such that $0<\rhodown{\rho}<\rhoup{\rho}<1$. Using \eqref{e}
\begin{align*}
	\nu^{1-\rhoup{\rho},1-\rhodown{\rho}}(d_i=s_i \qquad \text{for $1 \leq i \leq \xi_1 M$})& = \nu^{1-\rhoup{\rho},1-\rhodown{\rho}}(e_i=0 \qquad \text{for $1 \leq i \leq \xi_1 M$})\\
	&=\nu^{1-\rhoup{\rho},1-\rhodown{\rho}}\Big(\sum_{i=1}^{\xi_1 M} e_i=0\Big).
\end{align*}
It follows that 
\begin{align}\label{C1}
	\P(C^{\xi,M}_1)=\nu^{1-\rhoup{\rho},1-\rhodown{\rho}}\Big(\sum_{i=1}^{\xi_1 M} e_i=0\Big),
\end{align}
and similarly
 \begin{align}\label{C2}
 \P(C^{\xi,M}_2)=\nu^{\rhodown{\rho},\rhoup{\rho}}\Big(\sum_{i=1}^{\xi_1 M} e_i=0\Big).
 \end{align}
 Altogether, plugging \eqref{C1} and \eqref{C2} into \eqref{C} we obtain
 \begin{align}
 	\P\big((C^{\xi,M})^c\big)\leq \nu^{1-\rhoup{\rho},1-\rhodown{\rho}}\Big(\sum_{i=1}^{\xi_1 M} e_i>0\Big)+\nu^{\rhodown{\rho},\rhoup{\rho}}\Big(\sum_{i=1}^{\xi_1 M} e_i>0\Big).
 \end{align}
 Let us now try to explain the idea behind the proof. Let $x_j=s_{j-1}-a_j$, from \eqref{wr} we see that 
 \begin{align*}
 	w_j=\big(w_{j-1}+x_j\big)^+.
 \end{align*}
 Define the stopping time
 \begin{align*}
 	T=\sup\big\{k:k>0,w_{k-1}+x_k\geq 0\big\}
 \end{align*}
 so that
 \begin{align*}
 	w_j=w_{j-1}+x_j \quad 1\leq j \leq T.
 \end{align*}
 Using this recursion and  \eqref{S}
 \begin{align}\label{wp}
 	&w_j=w_0+S^{1,j} \quad 1\leq j \leq T\\
 \text{and}\quad	&w_j\geq 0 \quad  1\leq j \leq T.\nonumber
 \end{align}
 The dynamics behind \eqref{wp} is as follows. The waiting time $w_j$ increases when the service times are longer then usual and the interarrival times are shorter i.e. when the random walk $S^{0,j}$ goes up. Similarly, the $w_j$ decreases when the service times are fast compared to the arrival of customers i.e. $S^{0,j}$ goes down. This dynamics hold until the random walk goes below $-w_0$ where the waiting time at the queue vanishes. The r.v. $\sum_{i=1}^{\xi_1 M} e_i$ can be thought of as the local time of the queue at zero, i.e. the accumulated time of the queue being empty. The main idea behind the  proof of Proposition \ref{prop:Cub} is the observation that when $\rhoup{\rho}-\rhodown{\rho}\sim N^{-\frac{1}{3}}$, that is when the queue is in the so-called heavy traffic regime, at stationarity, the waiting time $w_0$ of  customer $0$, is of order $N^\frac{1}{3}$. As the difference between the average service time rate and the average inter-arrival time rate is of order $\rhoup{\rho}-\rhodown{\rho}\sim N^{-\frac{1}{3}}$, the simple random walk $S^{0,j}$ has  drift $-N^{-\frac{1}{3}}$.  \eqref{wp} implies that the queue's waiting time vanishes by time of order $N^\frac{2}{3}$. Over time $t=o(N^{\frac{2}{3}})$ the random walk $S_t$ will not change the waiting time at the queue by much so that with high probability $w_t$ will be of order $N^{\frac{1}{3}}$ and the r.v. $\sum_{i=1}^{t} e_i$  will be zero (see Figure \ref{fig:emptyq}).
\begin{lemma}
	\begin{figure}[t]
		\centering	
		\begin{subfigure}{.4\textwidth}
			\includegraphics[scale=1]{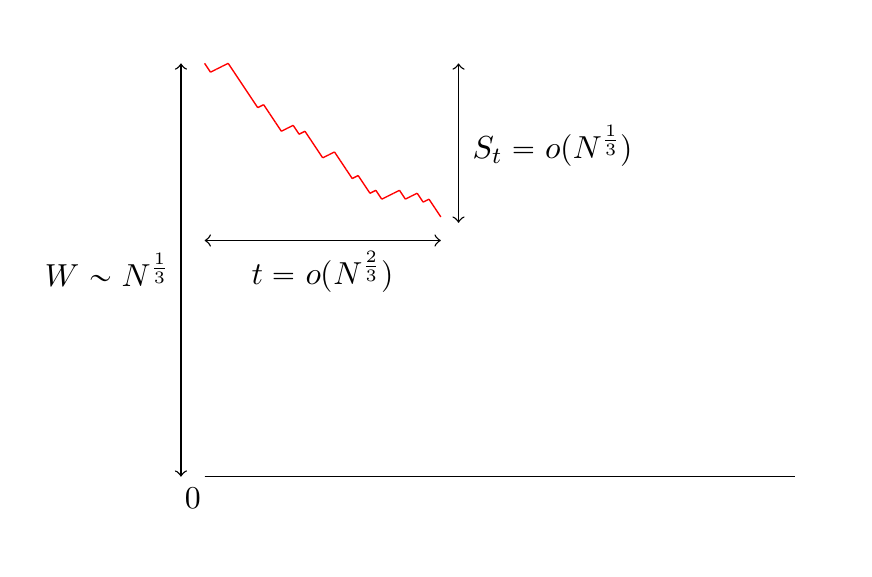}
			\caption{\small While the waiting time $W$  at the queue is of order $N^{\frac{1}{3}}$, over time of order smaller than $N^\frac{2}{3}$ the waiting time is not likely to change by much.}
		\end{subfigure}\hspace{2cm}%
		\begin{subfigure}{.4\textwidth}
		\includegraphics[scale=1]{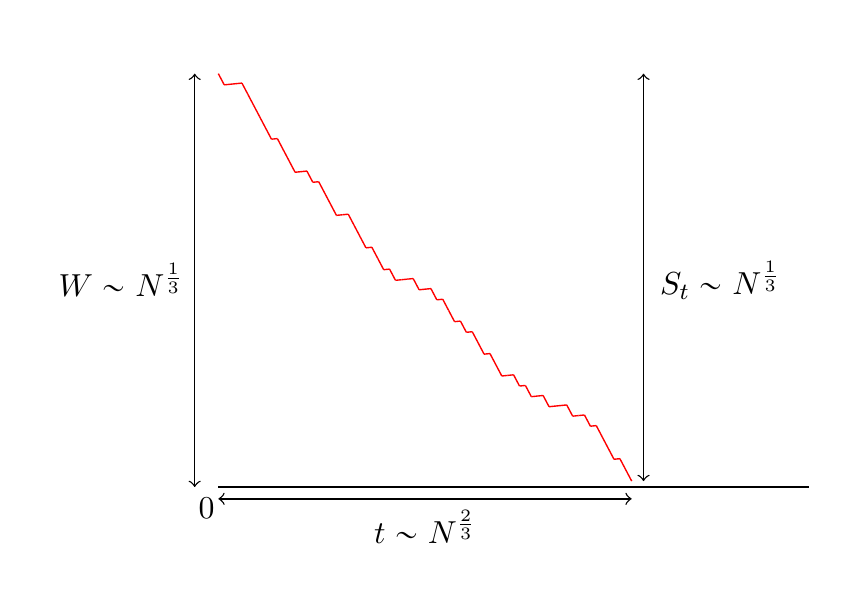}
		\caption{\small Over time of order  $N^\frac{2}{3}$, with positive probability  the waiting time $W$ vanishes, i.e. the queue will be empty.}
		\end{subfigure}
		\caption{\small The two cases of a queue at stationarity. $S_t$ is the random walk whose incremental step is $x_t=s_{t-1}-a_t$. As the rate of service at the queue is higher than the rate of interarrival $\E(x_t)<0$ and so $S_t$ is a simple random walk with a negative drift. The waiting time at the queue decreases by $S_t$ until it vanishes.} 
		\label{fig:emptyq}
	\end{figure}
	Let $\xi\in \ri\Uset$ and let $M>0$. For  $0<\beta<\alpha<1$ 
	\begin{align}\label{ub}
		\nu^{\beta,\alpha}\Big(\sum_{i=1}^{\xi_1 M} e_i>0\Big)\leq 1-\frac{\beta}{\alpha}+\int\Big[\frac{\alpha}{(\alpha+\theta)}\frac{\beta}{(\beta-\theta)}\Big]^{\xi_1M}e^{-\theta w}\frac{(\alpha-\beta)\beta}{\alpha}e^{-(\alpha-\beta)w}dw. 
	\end{align}
\end{lemma}
\begin{proof}
	By \eqref{sume} 
	\begin{align}\label{se2}
		\sum_{i=1}^{\xi_1 M} e_i=\Big(\inf_{1\leq i\leq \xi_1M}w_{0}+S_x^{1,i}\Big)^-
	\end{align}
	where
	\begin{align*}
		S^{1,i}_x=\sum_{j=1}^{i}s_{j-1}-a_j.
	\end{align*}
	Next we  bound from above the probability that the infimum of the path of  $\{S_x^{1,i}\}_{1 \leq i\leq \xi_1 M}$ drops too low. Let $C>0$, then
	\begin{align*}
		\P\Big(\inf_{1\leq i\leq \xi_1M}S^{1,i}_x\leq -C\Big)=\P\Big(\sup_{1\leq i\leq \xi_1M}-S^{1,i}_x\geq C\Big).
	\end{align*}
	As $-S^{1,i}_x$ is a  submartingale and $\phi_\theta(x)=e^{\theta x}$ is a strictly increasing convex function for $\theta>0$ $\phi_{\theta}(-S^{1,i}_x)$ is again a submartingale. By Doob's inequality
	\begin{align}\label{doobi}
		\P\Big(\sup_{1\leq i\leq \xi_1M}-S^{1,i}_x\geq C\Big)=\P\Big(\sup_{1\leq i\leq \xi_1M}e^{\theta(-S^{1,i}_x)}\geq e^{\theta C}\Big)\leq \frac{\E\big(e^{\theta(-S^{1,\xi_1M}_x)}\big)}{e^{\theta C}}
	\end{align}
	Note that by the independence of $\arrv=(\arr_j)_{j\in\Z}$ and  $\servv=(\serv_j)_{j\in\Z}$, for $-\alpha<\theta<\beta$ 
	\begin{align}\label{mgf}
		\E\Big(e^{\theta(-S^{1,\xi_1 M})}\Big)=\Big[\E\Big(e^{-\theta s_1}\Big)\E\Big(e^{\theta a_1}\Big)\Big]^{\xi_1 M}=\Big[\frac{\alpha}{(\alpha+\theta)}\frac{\beta}{(\beta-\theta)}\Big]^{\xi_1M}.
	\end{align}
	Plugging \eqref{mgf} in \eqref{doobi}
	\begin{align}\label{doobi2}
		\P\Big(\sup_{1\leq i\leq \xi_1M}-S^{1,i}_x\geq C\Big)\leq\Big[\frac{\alpha}{(\alpha+\theta)}\frac{\beta}{(\beta-\theta)}\Big]^{\xi_1M}e^{-\theta C}
	\end{align}
	By \eqref{se2}
	\begin{align*}
	\sum_{i=1}^{\xi_1 M} e_i>0 \iff w_{0}+\inf_{1\leq i\leq \xi_1M}S_x^{1,i}<0,
	\end{align*}
	and so
	\begin{align}
		\P\Big(\sum_{i=1}^{\xi_1 M} e_i>0\Big)=\P\big(w_{0}+\inf_{1\leq i\leq \xi_1M}S_x^{1,i}<0\big)
	\end{align}
	Note that by the definition of $w_0$ (\eqref{w}), $w_0$ is independent of  $\{S_x^{1,i}\}_{i\in \Z_{> 0}}$ and so
	\begin{align}
		&\P\Big(w_{0}+\inf_{1\leq i\leq \xi_1M}S_x^{1,i}<0\Big)=\P\Big(\sup_{1\leq i\leq \xi_1M}-S_x^{1,i}>w_{0}\Big)\nonumber\\\nonumber
		&=\int\P\Big(\sup_{1\leq i\leq \xi_1M}-S_x^{1,i}>w|w_0=w\Big)\P(w_0\in dw)\\\label{eq}
		&=\int\P\Big(\sup_{1\leq i\leq \xi_1M}-S_x^{1,i}>w|w_0=w\Big)f_w(dw)=\int\P\Big(\sup_{1\leq i\leq \xi_1M}-S_x^{1,i}>w\Big)f_w(dw)
	\end{align}
	where $f_w$ is given by (see \eqref{des})
	\begin{align}\label{fw}
		f_w(dw)=\Big((1-\frac{\beta}{\alpha})\delta_0(dw)+\frac{(\alpha-\beta)\beta}{\alpha}e^{-(\alpha-\beta)w}dw\Big),
	\end{align}
	so that
	 \begin{align}\label{rwub}
	 	\P\Big(w_{0}+\inf_{1\leq i\leq \xi_1M}S_x^{1,i}<0\Big)\leq 1-\frac{\beta}{\alpha}+\int\P\Big(\sup_{1\leq i\leq \xi_1M}-S_x^{1,i}>w\Big)\frac{(\alpha-\beta)\beta}{\alpha}e^{-(\alpha-\beta)w}dw.
	 \end{align}
	  Plugging \eqref{doobi2} in \eqref{rwub} we obtain the result.
\end{proof}
\medskip 
\begin{lemma}
	Let $\xi\in \ri\Uset$ and let $M>0$. For $\rhodown{\rho}(\xi)=\rho(\xi)-rN^{-\frac{1}{3}}$, $\rhoup{\rho}(\xi)=\rho(\xi)+rN^{-\frac{1}{3}}$ and $0<\theta<\rhoup{\rho}$,
	\begin{align}\label{tub}
	\nu^{\rhodown{\rho},\rhoup{\rho}}\Big(\sum_{i=1}^{\xi_1 M} e_i>0\Big)\leq \frac{2rN^{-\frac{1}{3}}}{\rho+rN^{-\frac{1}{3}}}+\frac{\rho-rN^{-\frac{1}{3}}}{\rho+rN^{-\frac{1}{3}}}\Bigg[1+\frac{2r\theta N^{-\frac{1}{3}}+\theta^2}{\rho^2-(r^2N^{-\frac{2}{3}}+2rN^{-\frac{1}{3}}\theta+\theta^2)}\Bigg]^{\xi_1M}\Big(1+2\theta r^{-1}N^\frac{1}{3}\Big)^{-1}.
	\end{align}
\end{lemma}
\begin{proof}
	Set $\beta=\rhodown{\rho}$ and $\alpha=\rhoup{\rho}$ so that 
	\begin{align}
	\Big[\frac{\alpha}{(\alpha+\theta)}\frac{\beta}{(\beta-\theta)}\Big]^{\xi_1M}e^{-\theta w} &\leq \Big[\frac{\rho+rN^{-\frac{1}{3}}}{(\rho+rN^{-\frac{1}{3}}+\theta)}\frac{\rho-rN^{-\frac{1}{3}}}{(\rho-rN^{-\frac{1}{3}}-\theta)}\Big]^{\xi_1M}e^{-\theta w}\nonumber\\
	&=\Big[\frac{\rho^2-r^2N^{-\frac{2}{3}}}{\rho^2-(r^2N^{-\frac{2}{3}}+2rN^{-\frac{1}{3}}\theta+\theta^2)}\Big]^{\xi_1M}e^{-\theta w}\nonumber\\
	&=\Big[1+\frac{2r\theta N^{-\frac{1}{3}}+\theta^2}{\rho^2-(r^2N^{-\frac{2}{3}}+2rN^{-\frac{1}{3}}\theta+\theta^2)}\Big]^{\xi_1M}e^{-\theta w}\label{ub2}.
	\end{align}
	and that 
\begin{align}\label{fw2}
	\frac{(\alpha-\beta)\beta}{\alpha}e^{-(\alpha-\beta)w}=\frac{(2rN^{-\frac{1}{3}})(\rho-rN^{-\frac{1}{3}})}{\rho+rN^{-\frac{1}{3}}}e^{-2rN^{-\frac{1}{3}}w}.
\end{align}
Using \eqref{ub2} and \eqref{fw2} in \eqref{ub} and the change of variable $2rN^{-\frac{1}{3}}w\mapsto w$
\begin{align}
	\nu^{\rhodown{\rho},\rhoup{\rho}}\Big(\sum_{i=1}^{\xi_1 M} e_i>0\Big)
	&=\frac{2rN^{-\frac{1}{3}}}{\rho+rN^{-\frac{1}{3}}}+\frac{\rho-rN^{-\frac{1}{3}}}{\rho+rN^{-\frac{1}{3}}}\int_0^\infty\Bigg[1+\frac{2r\theta N^{-\frac{1}{3}}+\theta^2}{\rho^2-(r^2N^{-\frac{2}{3}}+2rN^{-\frac{1}{3}}\theta+\theta^2)}\Bigg]^{\xi_1M}e^{-\theta \frac{1}{2r}N^{\frac{1}{3}}w}e^{-w}dw
	\nonumber\\
	&=\frac{2rN^{-\frac{1}{3}}}{\rho+rN^{-\frac{1}{3}}}+\frac{\rho-rN^{-\frac{1}{3}}}{\rho+rN^{-\frac{1}{3}}}\Bigg[1+\frac{2r\theta N^{-\frac{1}{3}}+\theta^2}{\rho^2-(r^2N^{-\frac{2}{3}}+2rN^{-\frac{1}{3}}\theta+\theta^2)}\Bigg]^{\xi_1M}\Big(1+\theta 2r^{-1}N^\frac{1}{3}\Big)^{-1}\nonumber
\end{align}
\end{proof}
\begin{lemma}\label{lem:C1ub}
	Let $\xi\in \ri\Uset$ and let $M>0$ such that $M\leq cN^\frac{2}{3}$. There exists $C(\xi,c)>0$ such that
	\begin{align}\label{C1ub}
		\P\Big(\big(C^{\xi,M}_1\big)^c\Big)\leq C N^{-\frac{1}{4}}(\xi_1M)^\frac{3}{8}. 
	\end{align}
\end{lemma}
\begin{proof}
	Let $r=(\xi_1M)^{-\frac{1}{8}}N^{\frac{1}{12}}$ and $\theta=(\xi_1M)^{-\frac{1}{2}}$ in \eqref{tub} to obtain
	\begin{align}\label{mub}
		\nu^{\rhodown{\rho},\rhoup{\rho}}\Big(\sum_{i=1}^{\xi_1 M} e_i>0\Big)\leq A+B,
	\end{align}
	where
	\begin{align*}
		A=\frac{2N^{-\frac{1}{4}}(\xi_1M)^{-\frac{1}{8}}}{\rho+N^{-\frac{1}{4}}(\xi_1M)^{-\frac{1}{8}}}
	\end{align*}
	and 
	\begin{align}\label{BT}
		B=B_1\times B_2 \times B_3
	\end{align}
	where
	\begin{align}
		B_1&=\frac{\rho-N^{-\frac{1}{4}}(\xi_1M)^{-\frac{1}{8}}}{\rho+N^{-\frac{1}{4}}(\xi_1M)^{-\frac{1}{8}}}=1-\frac{2N^{-\frac{1}{4}}(\xi_1M)^{-\frac{1}{8}}}{\rho+N^{-\frac{1}{4}}(\xi_1M)^{-\frac{1}{8}}}\label{B1}\\
		B_2&=\Bigg[1+\frac{2N^{-\frac{1}{4}}(\xi_1M)^{-\frac{5}{8}}+(\xi_1M)^{-1}}{\rho^2-(N^{-\frac{1}{2}}(\xi_1M)^{-\frac{1}{4}}+2N^{-\frac{1}{4}}(\xi_1M)^{-\frac{5}{8}}+(\xi_1M)^{-1})}\Bigg]^{\xi_1M}\label{B22}\\
		B_3&=\Big(1+2(\xi_1M)^{-\frac{3}{8}}N^{\frac{1}{4}}\Big)^{-1}\leq N^{-\frac{1}{4}}(\xi_1M)^\frac{3}{8}.\label{B3}
	\end{align}
	There exists $C_A(\rho)>0$ such that 
	\begin{align}\label{A}
		A\leq C_AN^{-\frac{1}{4}}(\xi_1M)^{-\frac{1}{8}} \quad \text{for $N\geq 1$}.
	\end{align}
	Note that by our assumption on $M$ the numerator in \eqref{B22} is dominated by $2cN^\frac{2}{3} \vee(\xi_1M)^{-1}$ and 
	\begin{align*}
		B_2 \rightarrow C(c) \quad \text{as $N\rightarrow \infty$},
	\end{align*}
	where $C(c)>0$ is locally bounded in  $c$. In particular, there exists $C_{B_2}(\rho)>0$ such that
	\begin{align}\label{B2}
		B_2\leq C_{B_2} \quad \text{for $N\geq 1$}.
	\end{align}
	Plugging \eqref{B1}, \eqref{B2} and \eqref{B3} into \eqref{BT} we see that there exists $C_B(\rho)>0$ such that
	\begin{align}\label{B}
		B\leq N^{-\frac{1}{4}}(\xi_1M)^\frac{3}{8}.
	\end{align}
	Plugging now \eqref{A} and \eqref{B} into \eqref{mub} and using \eqref{C1} we obtain the result.
\end{proof}
\begin{proof}[Proof of Proposition \ref{prop:Cub}]
	Similar to Lemma \ref{lem:C1ub} one can show that for $\xi\in \ri\Uset$ and  $M>0$ such that $M\leq cN^\frac{2}{3}$. There exists $C(\xi,c)>0$ such that
	\begin{align}\label{C2ub}
	\P\Big(\big(C^{\xi,M}_2\big)^c\Big)\leq C N^{-\frac{1}{4}}(\xi_2M)^\frac{3}{8}. 
	\end{align}
	\eqref{C1ub} and \eqref{C2ub} imply the result.
\end{proof}
\begin{proof}[Proof of Theorem \ref{thm:loc}]
	Plugging \eqref{ubA} and \eqref{ubC} into \eqref{ubs} we see that there exists $c_0>0$ such that for every  $c\leq c_0$
	\begin{align}\label{ine}
		\P\Big((\mathcal{H}^{\xi,cN^\frac{2}{3}})^c\Big)\leq CN^{-\frac{1}{4}}(cN^\frac{2}{3})^\frac{3}{8}\leq Cc^\frac{3}{8}.
	\end{align}
	By the definition of $\mathcal{H}^{\xi,cN^\frac{2}{3}}$, \eqref{ine} shows that there exists a coupling between 
	\begin{align*}
			\tilde{H}^{N,c}_{(x,y)}=\rim{G}_{N\xi,y}-\rim{G}_{N\xi,x} \quad (x,y)\in \mathcal{E}(R^{\xi,c}),
	\end{align*}
	and $B^{\xi(\rho)}|_{\mathcal{E}(R^{\xi,c})}$, the Busemann function $B^{\xi(\rho)}$ restricted on edges in $\mathcal{E}(R^{\xi,c})$, such that
	\begin{align}\label{ine2}
		\P\Big(\tilde{H}^{N,c}\neq B^{\xi(\rho)}|_{\mathcal{E}(R^{\xi,c})}\Big)\leq Cc^\frac{3}{8}. 
	\end{align}
	\eqref{ine2} shows that
	\begin{align}\label{ine3}
		d_\text{TV}\Big(\tilde{H}^{N,c}, B^{\xi(\rho)}|_{\mathcal{E}(R^{\xi,c})}\Big)\leq Cc^\frac{3}{8}.
	\end{align}  
	As the distribution of $B^{\xi(\rho)}|_{\mathcal{E}(R^{\xi,c})}$ equals that of 
	\begin{align*}
		\tilde{H}^{{\xi(\rho)},N,c}_{(x,y)}=\rim{G}^{\rho(\xi)}_{N\xi,y}-\rim{G}^{\rho(\xi)}_{N\xi,x} \quad (x,y)\in \mathcal{E}(R^{\xi,c}),
	\end{align*}
	\eqref{ine3} implies that 
	\begin{align}\label{ine4}
		d_\text{TV}\Big(\tilde{H}^{N,c}, \tilde{H}^{{\xi(\rho)},N,c}\Big)\leq Cc^\frac{3}{8}.
	\end{align}
	\eqref{ine4} implies \eqref{loc} by rotating the LPP picture by $180^\circ$.
\end{proof}
\begin{proof}[Proof of Theorem \ref{thm:stb}]
	Plugging \eqref{ubA} and \eqref{ubC} into \eqref{ubs} we obtain the result. 
\end{proof}
\begin{proof}[Proof of Theorem \ref{thm:airy}]
	 Define	
	\begin{align*}
		&\Delta L^N_x=L^N_x-L^N_0\\
		&=2^{-\frac{4}{3}}N^{-\frac{1}{3}}\Big(G_{(0,0),(N+x(2N)^\frac{2}{3},N-x(2N)^\frac{2}{3})}-G_{(0,0),(N,N)}\Big).
	\end{align*}
	By Theorem \ref{thm:loc} there exists $c_0>0$ and  $C(c_0)>0$ such that for any $|c|\leq c_0$, with probability at least $1-Cc^\frac{3}{8}$, simultaneously for all $|x|\leq c$
	\begin{align*}
		G_{(0,0),(N+x(2N)^\frac{2}{3},N-x(2N)^\frac{2}{3})}-G_{(0,0),(N,N)}=G^\frac{1}{2}_{(0,0),(N+x(2N)^\frac{2}{3},N-x(2N)^\frac{2}{3})}-G^\frac{1}{2}_{(0,0),(N,N)}.
	\end{align*} 
	Defining 
	\begin{align*}
		\Delta L^{\frac12,N}_x=2^{-\frac{4}{3}}N^{-\frac{1}{3}}\Big(G^\frac{1}{2}_{(0,0),(N+x(2N)^\frac{2}{3},N-x(2N)^\frac{2}{3})}-G^\frac{1}{2}_{(0,0),(N,N)}\Big),
	\end{align*}
	we conclude that with probability at least $1-Cc^\frac{3}{8}$,  simultaneously for all $|x|\leq c$
	\begin{align}
		\Delta L^N_x=\Delta L^{\frac12,N}_x.
	\end{align}
	which implies that
	\begin{align}\label{eq1}
		d_{TV}\big(\Delta L^N_x|_{[-c,c]},\Delta L^{\frac12,N}_x|_{[-c,c]}\big)\leq Cc^{\frac{3}{8}}.
	\end{align}
	The following limits are in distribution in the topology of continuous functions on $[-c,c]$.
	\begin{align}
		&\lim_{N\rightarrow \infty}\Delta L^N=\mathcal{A}'_2=\mathcal{A}_2(x)-\mathcal{A}_2(0)-x^2\label{limits}\\
		&\lim_{N\rightarrow \infty}\Delta L^{\frac12,N}=\mathcal{B}\label{limits2}.
	\end{align}
	Using Lemma \ref{lem:cod2} with \eqref{limits}--\eqref{limits2} and \eqref{eq1} implies that
	\begin{align*}
		d_{TV}\big(\mathcal{A}'_2|_{[-c,c]},\mathcal{B}|_{[-c,c]}\big)\leq 3Cc^\frac{3}{8},
	\end{align*}
	which implies the result.
\end{proof}
\begin{proof}[Proof of Corollary \ref{thm:airyr}]
	By the stationarity of the $\mathcal{A}_2$, it is enough to verify the claim for $I=[0,a]$ for some $a>0$. For every $\epsilon>0$, let $\Omega^\epsilon=C[0,\epsilon]$ be the space of continuous functions on the interval $[0,\epsilon]$. Let $\mathcal{F}^\epsilon$ be the Borel sigma algebra associated with the supremum metric on $\Omega^\epsilon$.  By Theorem \ref{thm:airy}, for every $\delta>0$ there exists $0<\epsilon\leq a$ and a probability space $(\Omega^\epsilon,\P^\epsilon)$ such that
	\begin{align}\label{reg}
		\P^\epsilon\Big(\mathcal{A}'_2|_{[0,\epsilon]}=\mathcal{B}|_{[0,\epsilon]}\Big)>1-\delta.
	\end{align}
	 For $\epsilon\in(0,a]$, \eqref{reg} implies that with probability larger than $1-\delta$
	\begin{align}\label{reg3}
	&\sup_{t\in I}\limsup_{h\downarrow 0}\frac{\mathcal{A}'_2(t+h)-\mathcal{A}'_2(t)}{\omega_B(h)}\geq \sup_{t\in [0,\epsilon)}\limsup_{h\downarrow 0}\frac{\mathcal{A}'_2(t+h)-\mathcal{A}'_2(t)}{\omega_B(h)}\\
	&= \sup_{t\in [0,\epsilon)}\limsup_{h\downarrow 0}\frac{\mathcal{B}(t+h)-\mathcal{B}(t)}{\omega_B(h)}=1,\nonumber
	\end{align}
	where the last equality comes from L\'{e}vy's modulus of continuity \cite[Theorem 10.1]{morters2010brownian} and self-similarity of the Brownian motion. Taking $\delta\rightarrow 0$
	\begin{align*}
		\sup_{t\in I}\limsup_{h\downarrow 0}\frac{\mathcal{A}'_2(t+h)-\mathcal{A}'_2(t)}{\omega_B(h)}\geq 1 \quad \text{with probability $1$}.
	\end{align*}
	Note that
	\begin{align}\label{reg2}
		&\sup_{t\in I}\limsup_{h\downarrow 0}\frac{\mathcal{A}'_2(t+h)-\mathcal{A}'_2(t)}{\omega_B(h)}\\\nonumber
		&=\sup_{t\in I}\limsup_{h\downarrow 0}\frac{\mathcal{A}_2(t+h)-\mathcal{A}_2(t)}{\omega_B(h)}+\sup_{t\in I}\limsup_{h\downarrow 0}\frac{(t+h)^2-t^2}{\omega_B(h)}\\\nonumber
		&=\sup_{t\in I}\limsup_{h\downarrow 0}\frac{\mathcal{A}_2(t+h)-\mathcal{A}_2(t)}{\omega_B(h)}
	\end{align}
	Plugging \eqref{reg2} in \eqref{reg3} implies the result.
\end{proof}
\section{Coalescence of Geodesics}
In this section we prove Theorem \ref{thm:ubc} and Theorem \ref{thm:ubc2}. For technical reasons, namely the direction in which we send $v_n$ to infinity in \eqref{v:853.3}, we prove the results for a setup  that is a bit different, yet equivalent,  to the one in Figure \ref{fig:coal} (see Figure \ref{fig:ubc}).
\subsection{Upper bound on $\P(|\ct-\pc|\leq \alpha N)$}
Let $\rim{G}^\rhoup{\rho}$ and $\rim{G}^\rhodown{\rho}$ be the stationary LPP with $\rhoup{\rho}=\rho+rN^{-\frac{1}{3}}$ and $\rhodown{\rho}=\rho-rN^{-\frac{1}{3}}$ constructed through \eqref{Gr11}--\eqref{Gr12} with the boundary weights on the north-east boundaries of $\Rb^{\xi N}$ as in \eqref{BW} and the bulk weights $\{\omega_x\}_{x\in\Rb^{\xi N}}$.
Similarly to  \eqref{Adef} define
\begin{align*}
\rim{A}^{r}=\Big\{\rim{Z}^{\rhodown{\rho}(\xi)}_{q^1,\ct}<-aN^\frac{2}{3}\Big\}\bigcap \Big\{\rim{Z}^{\rhoup{\rho}(\xi)}_{q^1,\ct}>0\Big\}.
\end{align*}
Similarly to Corollary \ref{cor:ubA} we have
\begin{lemma}
	Fix $\xi\in \ri\Uset$ and $a>0$. There exist $C(\xi,a)>0$, locally bounded in $a$, and $N_0(\xi,r)>0$ such that 
	\begin{align}\label{ub3}
	\P((\rim{A}^{r})^c)\leq Cr^{-3}.
	\end{align}
\end{lemma}
\begin{proof}
	By definition of $\rim{A}^{r}$ 
	\begin{align}
	\P((\rim{A}^{r})^c)\leq \P\big(\rim{Z}^{\rhodown{\rho}(\xi)}_{q^1,\ct}\geq-aN^\frac{2}{3}\big)+\P\big(\rim{Z}^{\rhoup{\rho}(\xi)}_{q^1,\ct}<0\big)
	\end{align}
	The bound on $\P\big(\rim{Z}^{\rhoup{\rho}(\xi)}_{q^1,\ct}<0\Big)$ comes from \eqref{lb-2}, it remains to bound $\P\big(\rim{Z}^{\rhodown{\rho}(\xi)}_{q^1,\ct}\geq -aN^\frac{2}{3}\big)$. Let $u=(u_1,u_2)=\xi N-aN^\frac{2}{3}e_2$, and let $\rim{G}^{\rhodown{\rho},[q^1]}_{u,x}$ be the LPP induced by $\rim{G}^\rhodown{\rho}_{q^1,x}$ at $u$. 
	By Lemma \ref{app-lm1} we see that
	\begin{align}
	\P\big(\rim{Z}^\rhodown{\rho}_{q^1,x}\geq-aN^\frac{2}{3}\big)=\P\big(\rim{Z}^{[q^1]}_{u,x}\geq 0\big),
	\end{align} 
	where $\rim{Z}^{[q^1]}_{u,x}$ is the exit point of $\rim{G}^{\rhodown{\rho},[q^1]}_{u,x}$. Compute
	\begin{align*}
	&\big(u_2-\frac{\rhodown{\rho}^2}{(1-\rhodown{\rho})^2}u_1\big)-\ct_2=\xi_2N-\frac{\rhodown{\rho}^2}{(1-\rhodown{\rho})^2}\xi_1 N-aN^\frac{2}{3}\\
	&=\frac{\xi_2 N\big(1-2(\rho-rN^{-\frac{1}{3}})+(\rho-rN^{-\frac{1}{3}})^2\big)-\big(\rho^2-2\rho rN^{-\frac{1}{3}}+r^2N^{-\frac{2}{3}}\big)\xi_1N}{(1-\rhodown{\rho})^2}-aN^\frac{2}{3}\nonumber\\
	&=\frac{ N\big(\xi_2(1-\rho)^2-\xi_1\rho^2\big)+2r N^\frac{2}{3}\big(\xi_1\rho +\xi_2(1-\rho)\big)+ N^\frac{1}{3}r^2(\xi_2-\xi_1)}{(1-\rhodown{\rho})^2}-aN^\frac{2}{3}\\
	&=\frac{2(r-c'a) N^\frac{2}{3}\big(\xi_1\rho +\xi_2(1-\rho)\big)+ N^{-\frac{1}{3}}r^2(\xi_2-\xi_1)}{(1-\rhodown{\rho})^2}\nonumber,
	\end{align*}
	where
	\begin{align}
		c'=\frac{(1-\rhodown{\rho})^2}{2[\xi_1\rho+\xi_2(1-\rho)]}.
	\end{align}	
	It follows that there exists $N_0(\xi,r)$ such that for $N>N_0$
	\begin{align}
		\Big(u_2-\frac{\rhodown{\rho}^2}{(1-\rhodown{\rho})^2}u_1\Big)-\ct_2>\frac{(r-ca) N^\frac{2}{3}\big(\xi_1\rho +\xi_2(1-\rho)\big)}{(1-\rho)^2}\nonumber,
	\end{align}
	where 
	\begin{align}
		c=\frac{(1-\rho)^2}{4[\xi_1\rho+\xi_2(1-\rho)]}.
	\end{align}
	 It then follows by \cite{sepp-cgm-18}[Corollary 5.10] that  there exists a constant $C_1(\xi)>0$ such that 
	 \begin{align*}
	 	\P\big(\rim{Z}^{[q^1]}_{u,x}>0\big)\leq C_1(r-ca)^{-3},
	 \end{align*}
	 the proof is now complete.
\end{proof}
Let $0<\alpha<1$ and $o_\alpha=\alpha\xi N$. We define $\Ra=[\ct,o_\alpha]$ to be the rectangle whose left bottom corner is $\ct$ and whose upper right corner is $o_\alpha$. We shall need the following result.
\begin{lemma}\label{lem:zub}
	Fix $\xi\in \ri\Uset$, $0<\alpha<1$ and $r>0$. There exists $C(\xi)>0$ such that for $t>\alpha r$ and $N>N_0(\xi,r)$
	\begin{align}
		&\P\Big(|\rim{Z}^\rhoup{\rho}_{o_\alpha,\ct}|\geq  tN^\frac{2}{3}\Big)\leq C\alpha^2 t ^{-3}\label{aub1}\\
		&\P\Big(|\rim{Z}^\rhodown{\rho}_{o_\alpha,\ct}|\geq  tN^\frac{2}{3}\Big)\leq C\alpha^2 t^{-3}\label{aub2}.
	\end{align}
\end{lemma}
\begin{proof}
	We prove \eqref{aub1} as \eqref{aub2} is similar. In fact we only prove here the upper bound for $\P\big(\rim{Z}^\rhoup{\rho}_{o_\alpha,\ct}>tN^\frac{2}{3}\big)$ as the bound on $\P\big(\rim{Z}^\rhoup{\rho}_{o_\alpha,\ct}<tN^\frac{2}{3}\big)$ is similar. Let $\rim{G}^{\rhoup{\rho},[o_\alpha]}_{u,x}$ be the LPP induced by $\rim{G}^\rhoup{\rho}_{o_\alpha,x}$ at $u$ where $u=\alpha \xi  N-A_1tN^\frac{2}{3}e_1$, and
	\begin{align}\label{A1}
	A_1=\frac{4(\xi_1\rho+\xi_2(1-\rho))}{\rho^2}.
	\end{align}
	By Lemma \ref{app-lm1} we see that
	\begin{align}\label{ib}
		\P\big(\rim{Z}^{[o_\alpha]}_{u,x}>0\big)=\P\big(\rim{Z}^\rhoup{\rho}_{o_\alpha,x}>A_1tN^\frac{2}{3}\big),
	\end{align}
	where $\rim{Z}^{[o_\alpha]}_{u,x}$ and $\rim{Z}^\rhoup{\rho}_{o_\alpha,x}$ are the exit points of $\rim{G}^{\rhoup{\rho},[o_\alpha]}_{u,x}$ and $\rim{G}^\rhoup{\rho}_{o_\alpha,x}$ respectively. We would like to show that the characteristic $\xi(\rhoup{\rho})$ emanating from the point $u=(u_1,u_2)$ goes well above the point $\ct$ on the scale of $N^\frac{2}{3}$. Compute
	\begin{align*}
		&\big(u_2-\frac{\rhoup{\rho}^2}{(1-\rhoup{\rho})^2}u_1\big)-\ct_2=\alpha\xi_2N-\frac{\rhoup{\rho}^2}{(1-\rhoup{\rho})^2}\big(\alpha\xi_1 N-A_1 tN^\frac{2}{3}\big)\\
		&=\frac{\alpha\xi_2 N\big(1-2(\rho+rN^{-\frac{1}{3}})+(\rho+rN^{-\frac{1}{3}})^2\big)-\big(\rho^2+2\rho rN^{-\frac{1}{3}}+r^2N^{-\frac{2}{3}}\big)\alpha\xi_1N}{(1-\rhoup{\rho})^2}+\frac{\rhoup{\rho}^2}{(1-\rhoup{\rho})^2}A_1 tN^\frac{2}{3}\nonumber\\
		&=\frac{\alpha N\big(\xi_2(1-\rho)^2-\xi_1\rho^2\big)-2r\alpha N^\frac{2}{3}\big(\xi_1\rho +\xi_2(1-\rho)\big)+\alpha N^\frac{1}{3}r^2(\xi_2-\xi_1)}{(1-\rhoup{\rho})^2}+\frac{\rhoup{\rho}^2}{(1-\rhoup{\rho})^2}A_1 tN^\frac{2}{3}\nonumber\\
		&\geq\frac{tN^{\frac{2}{3}}\Big(\rho^2A_1-\frac{2r\alpha}{t} \big(\xi_1\rho +\xi_2(1-\rho)\big)+\frac{\alpha r^2}{t} N^{-\frac{1}{3}}(\xi_2-\xi_1)\Big)}{(1-\rhoup{\rho})^2}\nonumber.
	\end{align*}
	By \eqref{A1}, for $t\geq r\alpha$
	\begin{align*}
	&\Big(u_2-\frac{\rhoup{\rho}^2}{(1-\rhoup{\rho})^2}u_1\Big)-\ct_2\geq\frac{t N^\frac{2}{3}\big(2\big(\xi_1\rho +\xi_2(1-\rho)\big)+\frac{\alpha r^2}{t} N^{-\frac{1}{3}}(\xi_2-\xi_1)\big)}{(1-\rho)^2}\nonumber.
	\end{align*}
	 It follows that there exists $N_0(r)>0$ such that for $N>N_0$
	\begin{align*}
		&\Big(u_2-\frac{\rhoup{\rho}^2}{(1-\rhoup{\rho})^2}u_1\Big)-\ct_2\geq\frac{\alpha^{-\frac{2}{3}}t (\alpha N)^\frac{2}{3}\big[\xi_1\rho +\xi_2(1-\rho)\big]}{(1-\rho)^2}\nonumber.
	\end{align*}
	  It then follows by \cite{sepp-cgm-18}[Corollary 5.10] that  there exists a constant $C'_1(\xi)>0$
	 \begin{align}\label{ib1}
	 	\P\big(\rim{Z}^{[o_\alpha]}_{u,x}>0\big)\leq C'_1\alpha^2 t^{-3}
	 \end{align}
	 Plugging \eqref{ib} in \eqref{ib1} implies that
	 \begin{align*}
	 	\P\big(\rim{Z}^\rhoup{\rho}_{o_\alpha,x}>A_1tN^\frac{2}{3}\big)\leq C'_1\alpha^2 t^{-3}.
	 \end{align*}
	  Applying the change of variables $A_1t\mapsto t$, there exists $C_1(\xi)>0$ such that
	 \begin{align}
	 \P\big(\rim{Z}^\rhoup{\rho}_{o_\alpha,x}>tN^\frac{2}{3}\big)\leq C_1\alpha^2 t^{-3}.
	 \end{align}
	 Similarly we show that there exists $C_2(\xi)$ such that
	 \begin{align*}
	 \P\big(\rim{Z}^\rhoup{\rho}_{o_\alpha,x}<-tN^\frac{2}{3}\big)\leq C_2\alpha^2 t^{-3}.
	 \end{align*}
	 Setting  $C=C_1\vee C_2$  implies the result.
\end{proof}
 Define the sets
\begin{align*}
\partial^{\alpha}&= \Big\{\alpha \xi N-ie_1\Big\}_{0\leq i \leq \alpha \xi_1 N}\bigcup \Big\{\alpha \xi N-ie_2\Big\}_{0\leq i \leq \alpha \xi_2 N}\\
\partial^{\alpha,t}_c&= \Big\{\alpha \xi N-ie_1\Big\}_{0\leq i \leq t N^\frac{2}{3}}\bigcup \Big\{\alpha \xi N-ie_2\Big\}_{0\leq i \leq t N^\frac{2}{3}}\\
\partial^{\alpha,t}_f&=\partial^{\alpha}\setminus \partial^{\alpha,t}_c.
\end{align*}
In words, $\partial^{\alpha}$ is the north-east boundary of $\Ra$, $\partial^{\alpha,t}_c$ are all the points in $\partial^{\alpha}$ whose $l_1$ distance from $o_\alpha$ is less or equal to $t N^\frac{2}{3}$ while $\partial^{\alpha,t}_f$ are the set of points in $\partial^{\alpha}$ whose $l_1$ distance from $o_\alpha$ is larger or equal to $t N^\frac{2}{3}$. Let $\rhoup{\pi}^{q^1,\ct}$ and  $\rhodown{\pi}^{q^1,\ct}$ be the geodesics that start from $q^1$ and terminate at $\ct$, associated to $\rim{G}^\rhoup{\rho}$ and $\rim{G}^\rhodown{\rho}$ respectively. 
Define
\begin{align}
	\cB^{r,\alpha,t}=\big\{\rhoup{\pi}^{q^1,\ct}\cap \partial^{\alpha,t}_f = \emptyset\big\}\cap \big\{\rhodown{\pi}^{q^1,\ct}\cap \partial^{\alpha,t}_f = \emptyset\big\}
\end{align}
The superscript $r$ in $\cB^{r,\alpha,t}$ appears implicitly in $\rhoup{\rho},\rhodown{\rho}$. The following result shows that with high probability the geodesics $\rhoup{\pi}^{q^1,\ct}$ and  $\rhodown{\pi}^{q^1,\ct}$ will not wonder too far from the point $o_\alpha$.
\begin{corollary}
	Fix $\xi\in \ri\Uset$, $0<\alpha<1$ and $r>0$. There exists $C(\xi)>0$ such that for $t>\alpha r$ and $N>N_0(\xi,r)$
	\begin{align}
		&\P\big((\cB^{r,\alpha,t})^c\big)\leq C\alpha^2 t^{-3}\label{iub}.
	\end{align}
\end{corollary}
\begin{proof}
	Note that 
	\begin{align}\label{es}
		&\Big\{\rhoup{\pi}^{q^1,\ct}\cap \partial^{\alpha,t}_f \neq \emptyset\Big\}=\Big\{|\rim{Z}^\rhoup{\rho}_{o_\alpha,\ct}|\geq  tN^\frac{2}{3}\Big\}\\
		&\Big\{\rhodown{\pi}^{q^1,\ct}\cap \partial^{\alpha,t}_f \neq \emptyset\Big\}=\Big\{|\rim{Z}^\rhodown{\rho}_{o_\alpha,\ct}|\geq  tN^\frac{2}{3}\Big\}\label{es2}.
	\end{align}
	Taking probabilities on both sides of \eqref{es} and \eqref{es2},  using Lemma \ref{lem:zub} and union bound we obtain \eqref{iub}. 
\end{proof}
Define the sets
\begin{align*}
\cD^{r,\alpha,t}_1&=\{B^{\rhoup{\xi}}_{o_\alpha-ke_1,o_\alpha-(k-1)e_1}=B^{\rhodown{\xi}}_{o_\alpha-ke_1,o_\alpha-(k-1)e_1}\}
\qquad\text{for $1 \leq k \leq t N^\frac{2}{3}$}\\
\cD^{r,\alpha,t}_2&=\{B^{\rhoup{\xi}}_{o_\alpha-ke_2,o_\alpha-(k-1)e_2}=B^{\rhodown{\xi}}_{o_\alpha-ke_2,o_\alpha-(k-1)e_2}\}
\qquad\text{for $1 \leq k \leq t N^\frac{2}{3}$}\\
\cD^{r,\alpha,t}&=\cD_1^{r,\alpha,t}\cap \cD_2^{r,\alpha,t},
\end{align*}
where the superscript $r$ is implicit in $\rhoup{\xi},\rhodown{\xi}$ (\eqref{xibar}).
\begin{lemma}
	For every $\xi\in\ri\Uset$ and $0<\alpha<1$, there exists $C(\xi)>0$ so that for every $r\geq 1$ and $t\leq r^{-2}$  there exists $N_0(r)>0$ such that for $N\geq N_0$
	\begin{align}\label{Dub}
	\P\Big((\cD^{r,\alpha,t})^c\Big)\leq Ct^{\frac{1}{2}}r.
	\end{align}
\end{lemma}
\begin{proof}
	We show \eqref{Dub} for $\cD^{r,\alpha,t}_2$ the result then follows by union bound. As in \eqref{C2} we have
	\begin{align}\label{D2ub}
		\P\Big(\big(\cD^{r,\alpha,t}_2\big)^c\Big)=\nu^{\rhodown{\rho},\rhoup{\rho}}\Big(\sum_{i=1}^{tN^\frac{2}{3}} e_i>0\Big).
	\end{align}
	Using \eqref{tub} with $\theta=t^{-\frac{1}{2}}N^{-\frac{1}{3}}$
	\begin{align}\label{qub}
		\nu^{\rhodown{\rho},\rhoup{\rho}}\Big(\sum_{i=1}^{tN^\frac{2}{3}} e_i>0\Big)\leq \frac{2rN^{-\frac{1}{3}}}{\rho+rN^{-\frac{1}{3}}}+\frac{\rho-rN^{-\frac{1}{3}}}{\rho+rN^{-\frac{1}{3}}}\Bigg[1+\frac{\big(2rt^{\frac{1}{2}}+1\big)t^{-1}N^{-\frac{2}{3}}}{\rho^2-(r^2N^{-\frac{2}{3}}+2rt^{-\frac{1}{2}}N^{-\frac{2}{3}}+t^{-1}N^{-\frac{2}{3}})}\Bigg]^{tN^\frac{2}{3}}\Big(1+2t^{-\frac{1}{2}} r^{-1}\Big)^{-1}
	\end{align}
	Sending $N$ to $\infty$, the right hand site of \eqref{qub} converges to
	\begin{align}\label{qub2}
	 e^{\rho^{-2}\big(2rt^{\frac{1}{2}}+1\big)}\Big(1+2t^{-\frac{1}{2}} r^{-1}\Big)^{-1}\leq e^{\rho^{-2}\big(2rt^{\frac{1}{2}}+1\big)}\frac{1}{2}t^{\frac{1}{2}} r.
	\end{align}
	Plugging \eqref{qub2} in \eqref{D2ub}, by our assumption on $t$, $rt^{\frac{1}{2}}\leq 1$, and so we see that there exists $C_2(\xi)>0$
	such that for every $r\geq 1$, there exists $N_0(r)>0$ such that for $N\geq N_0$
	\begin{align*}
	\P\Big((\cD_2^{r,\alpha,t})^c\Big)\leq C_2t^{\frac{1}{2}}r.
	\end{align*}
	Similar bound can be obtained for $\cD_2^{r,\alpha,t}$ the result then follows by union bound.
\end{proof}
\begin{figure}[t]
	\includegraphics[scale=1]{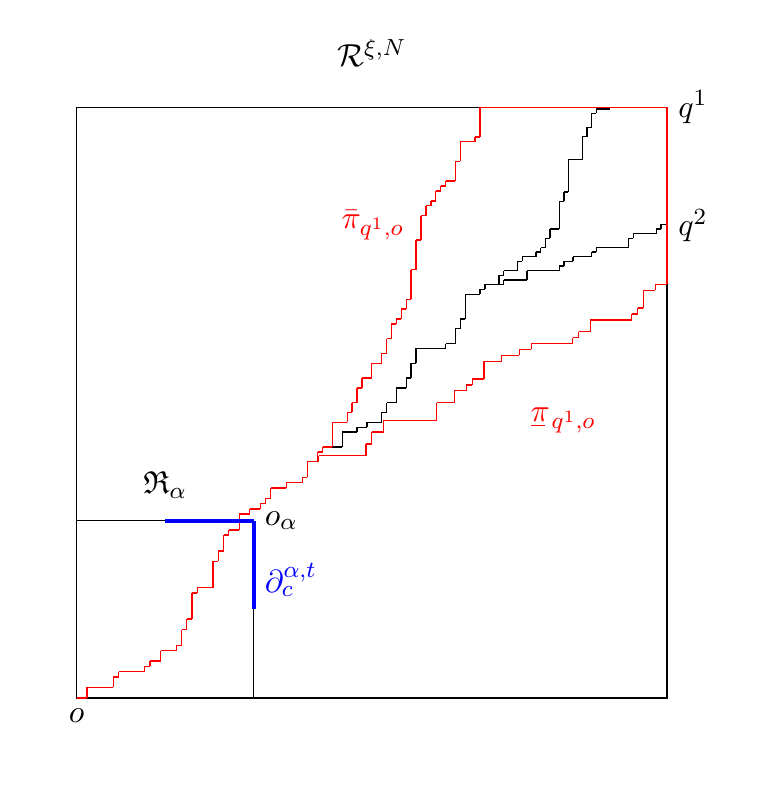}
	\caption{\small With high probability the geodesics $\rhoup{\pi}_{q^1,o}$ and $\rhodown{\pi}_{q^1,o}$ sandwich the geodesics $\pi_{q^1,o}$ and $\pi_{q^2,o}$. The stationary geodesics (in red) use the same weights on edges in $\mathcal{E}(\partial^{\alpha,t}_c)$.}\label{fig:ubc} 
\end{figure}
\begin{proof}[Proof of Theorem \ref{thm:ubc}]
 We first claim that on the event $\rim{A}^{r} \cap \cB^{r,\alpha,t}\cap \cD^{r,\alpha,t}$ the geodesics $\pi^{q^1,o}$ and $\pi^{q^2,o}$ must coalesce outside $\Ra$ (see Figure \ref{fig:ubc}). On the event $\rim{A}^{r}$
	\begin{align}
		\rhoup{\pi}_{q^1,o}\og \pi^{q^1,o} \og \pi^{q^2,o} \og \rhodown{\pi}_{q^1,o}.
	\end{align}
	This means that coalescence of the geodesics $\rhoup{\pi}_{q^1,o}$ and $\rhodown{\pi}_{q^1,o}$ outside $\Ra$ implies the coalescence of the geodesics $\pi^{q^1,o}$ and $\pi^{q^2,o}$ outside $\Ra$. It is therefore enough to show that on the set $\cB^{r,\alpha,t}\cap \cD^{r,\alpha,t}$ 
	\begin{align}\label{cog}
		\pr^{\xi,\alpha N}(\rhoup{\pi}_{q^1,o})=	\pr^{\xi,\alpha N}(\rhodown{\pi}_{q^1,o}).
	\end{align}
	On the event $\cB^{r,\alpha,t}$ the geodesics $\rhoup{\pi}_{q^1,o}$ and $\rhodown{\pi}_{q^1,o}$ do not cross $\partial^{\alpha,t}_f$ and therefore use only the weights $B^\rhoup{\rho}_e,B^\rhodown{\rho}_e$ where $e\in\mathcal{E}(\partial^{\alpha,t}_c)$ and the bulk weights $\{\omega_x\}_{x\in\Ra}$. It follows that on $\cB^{r,\alpha,t}\cap \cD^{r,\alpha,t}$ \eqref{cog} holds. 	Set $r=\alpha^{-\frac{2}{27}},t=\alpha^\frac{16}{27}$ so that $t=\alpha^{\frac{16}{27}}\geq \alpha^{\frac{25}{27}}=\alpha r$ holds (since $0<\alpha<1$). Use \eqref{ub3}, \eqref{iub} and \eqref{Dub} to see that there exists $C'(\xi,a)>0$ such that
	\begin{align*}
		&\P(|\ct-\pc|\leq (\xi_1\wedge\xi_2) \alpha N)\leq \P(\pc \in \Ra)\\
		&\leq  \P\Big((\rim{A}^{\alpha^{-\frac{2}{27}}})^c\Big)+\P\Big((\cB^{\alpha^{-\frac{2}{27}},\alpha,\alpha^{\frac{16}{27}}})^c\Big)+\P\Big((\cD^{\alpha^{-\frac{2}{27}},\alpha,\alpha^{\frac{16}{27}}})^c\Big)\\
		&\leq C'\Big(\alpha^{\frac{2}{9}}+\alpha^{\frac{2}{9}}+\alpha^{\frac{2}{9}}\Big).
	\end{align*}
	The result now follows.
\end{proof}
\subsection{Upper bound on $\P(|q^2-\pc|\leq \alpha N)$} 
For every $\xi\in \ri\Uset$ and $m\in \Z$, define the set
\begin{align*}
&\cC^{\xi+}_{m,t}=\{m\}\times \{y\in\Z:m\nicefrac{\xi_2}{\xi_1}-y\leq tN^{\frac{2}{3}}\}\\
&\cC^{\xi-}_{m,t}=\{m\}\times \{y\in\Z:m\nicefrac{\xi_2}{\xi_1}-y\geq -tN^{\frac{2}{3}}\}.
\end{align*}
Let $\xi^1=\xi$ and $\xi^2=\xi-(0,aN^{-\frac{1}{3}})$ be two vectors whose direction  is that of the characteristics emanating from $\ct$ associated with the point $q^1$ and $q^2$ respectively. Let $\rhoup{\rho}=\rho(\xi)+rN^{-\frac{1}{3}}$, $\rhodown{\rho}=\rho(\xi)-rN^{-\frac{1}{3}}$ and consider  $G^\rhoup{\rho}_{o,x}$ and $G^\rhodown{\rho}_{o,x}$ on $\Rb^{\xi N}$ as in \eqref{Gr2}.  For $x\in o+\Z^2_{>0}$, let $\rhoup{\pi}^{o,x}$ and $\rhodown{\pi}^{o,x}$ be the geodesics associated with the last passage time $G^\rhoup\rho_{o,x}$ and $G^\rhodown\rho_{o,x}$ respectively. We shall need the following auxiliary result.
\begin{lemma}\label{lem:cg}
	Let $0<\alpha<1$ and $r\geq 1$. There exists $N_0(\xi,r),C(\xi),A(\xi)>0$ such that for $N>N_0$  and $t\geq A\alpha r$ 
	\begin{align}
		&\P\Big(\pi^{q^1,o}\in (\cC^{\xi^1+}_{(1-\alpha)\xi_1 N,t})^c\Big)\leq C(\alpha^{2}t^{-3}+r^{-3})\label{cub2}\\
		&\P\Big(\pi^{q^2,o}\in (\cC^{\xi^2-}_{(1-\alpha)\xi_1 N,t})^c\Big)\leq C(\alpha^{2}t^{-3}+r^{-3})\label{cub3}.
	\end{align}
\end{lemma} 
\begin{proof}
	We prove only \eqref{cub2} as the proof of \eqref{cub3} is similar. 
	We would first like to show that there exist $N_0(\xi,r),C_1(\xi),A(\xi)>0$ such that for $N>N_0$  and $t\geq A\alpha r$  (see Figure \ref{fig:gd})
	\begin{align}\label{ubp}
		\P\Big(\rhoup{\pi}^{o,q^1}\in \big(\cC^{\xi+}_{(1-\alpha)\xi_1 N,t}\big)^c\Big) \leq C_1\alpha^{2}t^{-3} .
	\end{align}
	To see that \eqref{ubp} holds, let $u=\big((1-\alpha) \xi_1N,(1-\alpha) \xi_2N-tN^\frac{2}{3}\big)$ and consider  $G^{\rhoup{\rho},[o]}_{u,x}$. Note that
	\begin{align}\label{pze}
		\big\{\rhoup{\pi}^{o,q^1}\in \big(\cC^{\xi+}_{(1-\alpha)\xi_1 N,t}\big)^c\big\}=\big\{Z^{[o]}_{u,q^1}>0\big\}.
	\end{align}
	We compute
	\begin{align*}
		&\xi_2N-\big[(1-\alpha) \xi_2 N-tN^\frac{2}{3}+\frac{\rhoup{\rho}^2}{(1-\rhoup{\rho})^2}\alpha \xi_1N\big]\\
		&=\alpha\xi_2N-\frac{\rhoup{\rho}^2}{(1-\rhoup{\rho})^2}\alpha\xi_1N+tN^\frac{2}{3}\\
		&=\frac{\alpha\xi_2 N\big(1-2(\rho+rN^{-\frac{1}{3}})+(\rho+rN^{-\frac{1}{3}})^2\big)-\big(\rho^2+2\rho rN^{-\frac{1}{3}}+r^2N^{-\frac{2}{3}}\big)\alpha\xi_1N}{(1-\rhoup{\rho})^2}+tN^\frac{2}{3}\nonumber\\
		&=\frac{\alpha N\big(\xi_2(1-\rho)^2-\xi_1\rho^2\big)-2r\alpha N^\frac{2}{3}\big(\xi_1\rho +\xi_2(1-\rho)\big)+\alpha N^\frac{1}{3}r^2(\xi_2-\xi_1)}{(1-\rhoup{\rho})^2}+tN^\frac{2}{3}\nonumber\\
		&=\frac{tN^{\frac{2}{3}}\Big((1-\rhoup{\rho})^2-(1-\rho)^2+(1-\rho)^2-\frac{2r\alpha}{t} \big(\xi_1\rho +\xi_2(1-\rho)\big)+\frac{\alpha r^2}{t} N^{-\frac{1}{3}}(\xi_2-\xi_1)\Big)}{(1-\rhoup{\rho})^2}\nonumber.
	\end{align*}
	If 
	\begin{align*}
		t\geq \frac{5\alpha r(\xi_1\rho+\xi_2(1-\rho))}{(1-\rho)^2},
	\end{align*}
	then there for $C'(\xi)>0$
	\begin{align}\label{lh}
		\alpha\xi_2N-\frac{\rhoup{\rho}^2}{(1-\rhoup{\rho})^2}\alpha\xi_1N+tN^\frac{2}{3}&\geq \frac{tN^{\frac{2}{3}}\Big((1-\rhoup{\rho})^2-(1-\rho)^2+\frac{3}{5}(1-\rho)^2+ C'r N^{-\frac{1}{3}}(\xi_2-\xi_1)\Big)}{(1-\rhoup{\rho})^2},
	\end{align}
	so that there exists $N_0(\xi,r)$ such that for $N>N_0$ the left hand side of \eqref{lh} is greater of equal to 
	\begin{align*}
		\frac{tN^{\frac{2}{3}}\Big(\frac{1}{2}(1-\rho)^2\Big)}{(1-\rho)^2}\nonumber=\frac{1}{2}\alpha^{-\frac{2}{3}}t(\alpha N)^{\frac{2}{3}}.
	\end{align*}
	 It then follows by \cite{sepp-cgm-18}[Corollary 5.10] that  there exists a constant $C_1(\xi)>0$ such that such that for $N>N_0$
	\begin{align}\label{zub}
	\P\big(Z^{[o]}_{u,q^1}>0\big)\leq C_1\alpha^{2}t^{-3}.
	\end{align}
	\eqref{ubp} now follows from \eqref{zub}  using \eqref{pze}. Next we use \eqref{ubp} to obtain \eqref{cub2}. To see that, we first note that (similar to \eqref{lb-2}) there exists $C_2(\xi)>0$ such that
	\begin{align*}
		\P\big(Z^{\rhoup{\rho}}_{o,q^1}>0\big)\geq 1-C_2r^{-3}
	\end{align*}
	which implies that
	\begin{align}\label{og1}
		\P\big(\pi^{q^1,o}\og\rhoup{\pi}^{o,q^1})\geq 1-C_2r^{-3}.
	\end{align}
	Note that
	\begin{align}\label{sp}
		\{\pi^{o,q^1}\in \big(\cC^{\xi+}_{(1-\alpha)\xi_1 N,t}\big)^c\}\cap \big\{\pi^{q^1,o}\og\rhoup{\pi}^{o,q^1}\big\}\subset \{\rhoup{\pi}^{o,q^1}\in \big(\cC^{\xi+}_{(1-\alpha)\xi_1 N,t}\big)^c\}.
	\end{align}
	Taking probability in \eqref{sp} and using \eqref{ubp} and \eqref{og1} we arrive at \eqref{cub2}. 
\end{proof}
\begin{figure}[t]
	\includegraphics[scale=1]{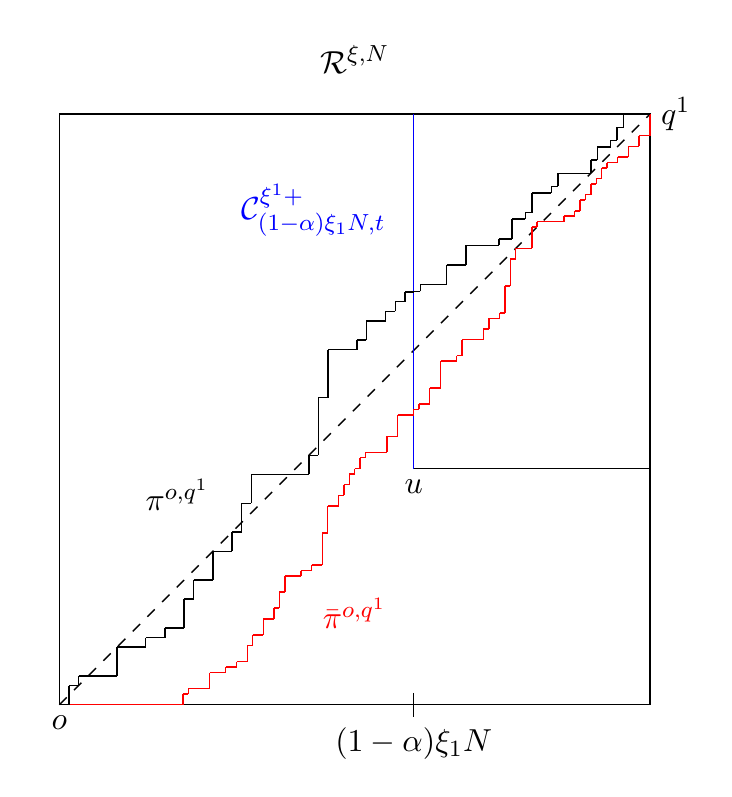}
	\caption{\small With high probability the geodesic $\rhoup{\pi}^{o,q^1}$ exits  from the south boundary of $\Rb^{\xi,N}$ and crosses the set $\cC^{\xi^1+}_{(1-\alpha)\xi_1 N,t}$.}
\end{figure}\label{fig:gd}
\begin{proof}[Proof of Theorem \ref{thm:ubc2}]
	Fix $0<\alpha<1$. Note that
	\begin{align*}
		 \Bigg(\frac{N\xi_2}{N\xi_1}-\frac{N\xi_2-aN^\frac{2}{3}}{N\xi_1}\Bigg)(1-\alpha) \xi_1 N=a(1-\alpha)N^\frac{2}{3}
	\end{align*}
	Let $t=\frac{a}{3}(1-\alpha)$ and  $r=\frac{t}{A}\alpha^{-\frac{2}{3}}$ where $A$ is the constant from Lemma \ref{lem:cg}, so that $t\geq A\alpha r=t\alpha^{\frac{1}{3}}$. By Lemma \ref{lem:cg}, there exists $C(\xi,a)>0$ 
	\begin{align}\label{aub}
	&\P\Big(\pi^{q^1,o}\in (\cC^{\xi^1+}_{(1-\alpha)\xi_1 N,t})^c\Big)\leq C\alpha^2\\
	&\P\Big(\pi^{q^2,o}\in (\cC^{\xi^2-}_{(1-\alpha)\xi_1 N,t})^c\Big)\leq C\alpha^2.\nonumber
	\end{align}
	 Let $p_1=\inf\{y:((1-\alpha)\xi_1 N,y)\in \pi^{q^1,o}\}$ and $p_2=\sup\{y:((1-\alpha)\xi_1 N,y)\in \pi^{q^2,o}\}$ be the lowest and highest  intersection points of the vertical line at $(1-\alpha)\xi_1 N$ with $\pi^{q^1,o}$ and $\pi^{q^2,o}$ respectively. \eqref{aub} implies that
	\begin{align*}
		\P\Big(p_1-p_2< \frac{a}{3}(1-\alpha)N^\frac{2}{3}\Big)\leq 2C\alpha^2.
	\end{align*}
	It follows that
	\begin{align}\label{ub1}
		\P\Big(p_c\in [\big((1-\alpha)\xi_1 N,0\big),\xi N]\Big)\leq 2C\alpha^2.
	\end{align}
	\eqref{ub1} implies the result.
	\begin{figure}[t]
		\includegraphics[scale=1]{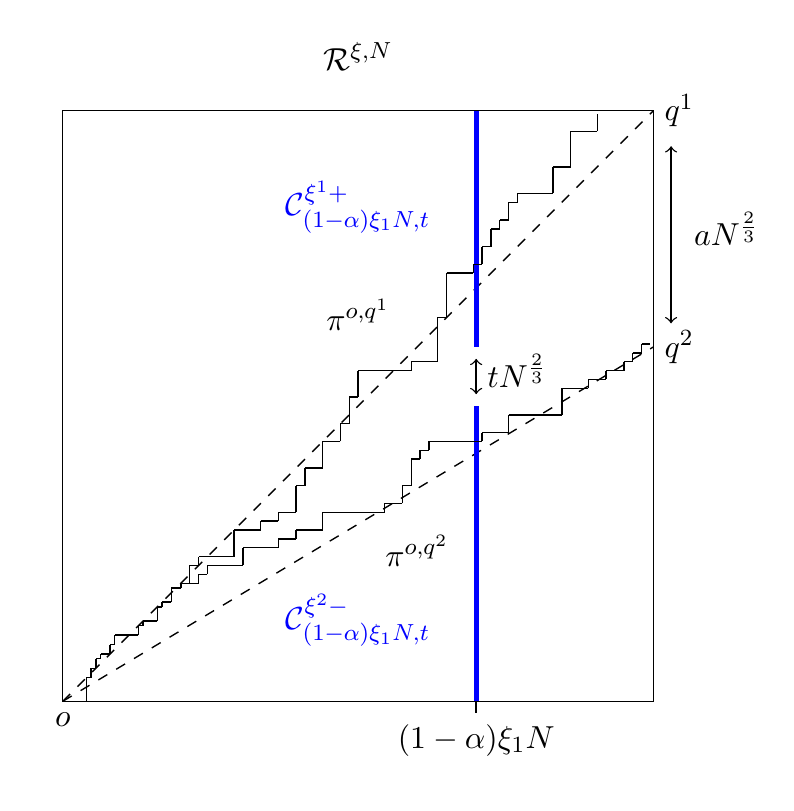}
		\caption{\small With high probability the geodesic $\pi^{o,q^1}$ crosses the vertical line at $(1-\alpha)\xi_1N$ no too far below the characteristic $\xi^1$  while $\pi^{o,q^2}$ crosses not too far above the characteristic $\xi^2$.}
	\end{figure}\label{fig:cg}
	
\end{proof}
\newpage
\appendix

\section{Queues} \label{app:queues}

We formulate  last-passage percolation over a bi-infinite strip as  a   queueing operator.   
The inputs are   two bi-infinite sequences: the {\it inter-arrival process}   $\arrv=(\arr_j)_{j\in\Z}$ and the {\it service process} $\servv=(\serv_j)_{j\in\Z}$.  
The queueing  interpretation is that $\arr_j$ is the time between the arrivals of customers $j-1$ and $j$ and    $\serv_j$ is the service time of customer $j$.  The operations below are well-defined as long as $\lim_{m\to-\infty}  \sum_{i=m}^0 (\serv_i-\arr_{i+1})=-\infty.$  

From  inputs $(\arrv,\servv)$  three output sequences 
\be\label{DSR}   \depav=\Dop(\arrv,\servv),  \quad \sojov=\Sop(\arrv,\servv), \quad\text{and} \quad \wc \servv=\Rop(\arrv,\servv) \ee
are constructed through explicit mappings: the {\it inter-departure process} $\depav=(\depa_j)_{j\in\Z}$, the {\it sojourn process} $\sojov=(\sojo_j)_{j\in\Z}$, and the {\it dual service times}  $\wc\servv=(\wc\serv_j)_{j\in\Z}$.  

The  formulas are as follows.  Choose a sequence  $G=(G_j)_{j\in\Z}$  that satisfies $\arr_j=G_j-G_{j-1}$.   Define the sequence  $\wt G=(\wt G_j)_{j\in\Z}$ by 
\be\label{m:800}
\wt G_j=\sup_{k:\,k\le j}  \Bigl\{  G_k+\sum_{i=k}^j \serv_i\Bigr\}. 
\ee
The supremum above  is taken  at some finite $k$.   Then set 
\be\label{DSR4} 
\depa_j =  \wt G_j - \wt G_{j-1}, \quad 
\sojo_j=\wt G_j -  G_j, \quad\text{and} \quad 
\wc \serv_j=\arr_j\wedge \sojo_{j-1}. \ee
The outputs \eqref{DSR4} are independent of the choice of $G$. Note that to compute $\{ \depa_j , \sojo_j, \wc \serv_j: j\le m\}$,  only inputs $\{ \arr_j , \serv_j : j\le m\}$ are needed. Let $\arrv=(\arr_j)_{j\in\Z}$ and $\servv=(\serv_j)_{j\in\Z}$ be two independent sequences of i.i.d. Exponential r.v. of intensity $\lambda$ and $\rho$ respectively, where $0<\lambda<\rho<1$. We denote by $\nu^{\lambda,\rho}$ the distribution of $(D(\arrv,\servv),\servv)$ on $\R_+^\Z\times \R_+^\Z$, i.e.
\begin{align}\label{qd}
	\nu^{\lambda,\rho}\sim (D(\arrv,\servv),\servv).
\end{align}  
By Burke's Theorem $D(\arrv,\servv)$ is sequences of i.i.d. Exponential r.v. of intensity $\lambda$, consequently, the measure $	\nu^{\lambda,\rho}$ is referred to as a stationary measure of the queue. The waiting time of the $j$'th customer is given by 
\begin{align}\label{w}
	w_j=\sup_{i\leq j}\Big(\sum_{k=i}^{j}s_{k-1}-a_k\Big)^+.
\end{align}
The random variables $\{w_j\}_{j\in\Z}$ satisfy 
\begin{align}\label{wr}
	w_j=\big(w_{j-1}+s_{j-1}-a_j\big)^+.
\end{align}
The distribution of $w_0$ (and by stationarity the distribution of any $w_j$ for $j\in \Z$) is given by
\begin{align}\label{des}
	f_w(dw)=\P(w_0\in dw)=\big(1-\frac{\lambda}{\rho}\big)\delta_0(dw)+\frac{(\rho-\lambda)\lambda}{\rho}e^{-(\rho-\lambda)w}dw.
\end{align}
One can write 
\begin{align}
	d_j=e_j+s_{j},
\end{align}
where $e_j$ is called the $j$'th \textit{idle time} and is given by
\begin{align}\label{e}
	e_j=(w_{j-1}+s_{j-1}-a_j)^-.
\end{align}
$e_j$ is the time between the departure of customer $j-1$ and the arrival of customer $j$ in which the sever is idle. Define 
\begin{align*}
	x_j=s_{j-1}-a_j,
\end{align*}
and the summation operator
\begin{align}\label{S}
	S^{k,l}=\sum_{i=k}^lx_i
\end{align}
Summing $e_j$ we obtain the cumulative idle time \cite{whit}[Chapter 9.2, Eq. 2.7] as in the following lemma.
\begin{lemma}
	For any $k\leq l$
\begin{align}\label{sume}
	\sum_{i=k}^l e_i=\Big(\inf_{k\leq i\leq l}w_{k-1}+S^{k,i}\Big)^-.
\end{align}
\begin{proof}
	For simpler exposition we set $k=1$. By \eqref{w}
	\begin{align*}
		w_l&=\Big(\sup_{1\leq i\leq l}S^{i,l}\Big)^+ \vee \Big(\sup_{-\infty<i\leq 0}S^{i,l}\Big)^+\\
		&=\Big(\sup_{1\leq i\leq l}S^{i,l}\Big)^+ \vee \Big(\sup_{-\infty<i\leq 0}S^{i,0}+S^{1,l}\Big)^+\\
		&=\Big(\sup_{1\leq i\leq l}S^{i,l}\Big)^+ \vee \Big(\big(\sup_{-\infty<i\leq 0}S^{i,0}\big)^++S^{1,l}\Big)^+.
	\end{align*}
	To see that the last equality holds, note that
	\begin{align*}
		w_l=
		\begin{cases}
		\Big(\sup_{1\leq i\leq l}S^{i,l}\Big)^+ \quad &\sup_{-\infty<i\leq 0}S^{i,0}<0\\
		\Big(\sup_{1\leq i\leq l}S^{i,l}\Big)^+ \vee \Big(\big(\sup_{-\infty<i\leq 0}S^{i,0}\big)^++S^{1,l}\Big)^+ \quad &\sup_{-\infty<i\leq 0}S^{i,0}\geq 0.
		\end{cases}
	\end{align*}
	It follows that
	\begin{align}\label{sume2}
		w_l&=\Big(\sup_{1\leq i\leq l}S^{i,l}\Big)^+ \vee \big(S^{1,l}+w_0\big)^+\\\nonumber
		&=\Big(\sup_{2\leq i \leq l}\big[S^{1,l}-S^{1,i-1}\big]\Big)^+\vee \big(S^{1,l}+w_0\big)^+\\\nonumber
		&=\Big(S^{1,l}-\inf_{1\leq i \leq l-1}S^{1,i}\Big)^+ \vee \big(S^{1,l}+w_0\big)^+\\\nonumber
		&=\Big(S^{1,l}+w_0-\inf_{0\leq i \leq l-1}\big[w_0+\hat{S}^{1,i}\big]\Big)^+ 
	\end{align}
	where
	\begin{align*}
		\hat{S}^{1,i}=
		\begin{cases}
		S^{1,i} \quad i>0\\
		-w_0 \quad i=0
		\end{cases},
	\end{align*}
	where the last equality in \eqref{sume2} follows since
	\begin{align*}
		\Big(S^{1,l}+w_0-\big[w_0+\hat{S}^{1,0}\big]\Big)^+=\big(S^{1,l}+w_0\big)^+.
	\end{align*}
	 It follows that
	\begin{align}\label{ew2}
		w_l=S^{1,l}+w_0-\inf_{0\leq i \leq l}\big[w_0+\hat{S}^{1,i}\big].
	\end{align}
	where we dropped the positive part because 
	\begin{align*}
	w_l\geq S^{1,l}+w_0-\big[w_0+\hat{S}^{1,l}\big]=0.
	\end{align*}
	By \eqref{wr} and \eqref{e} we see that
	\begin{align}\label{ew}
		-e_j+w_j&=w_{j-1}+s_{j-1}-a_j\\
		\implies e_j&=w_j-w_{j-1}-x_j.\nonumber
	\end{align}
	Summing on both sides of \eqref{ew} and using \eqref{ew2}
	\begin{align*}
		\sum_{i=1}^le_i&=w_l-w_0-S^{1,l}\\
		&=-\inf_{0\leq i \leq l}w_0+\hat{S}^{1,i}\\
		&=\Big(\inf_{1\leq i \leq l}w_0+S^{1,i}\Big)^-.
	\end{align*} 
\end{proof}
\end{lemma}  
The next lemma is a  deterministic property of the mappings. 

\begin{lemma}\label{lm:DR}   The identity 
	$\Dop\bigl(\Dop(\barrv, \arrv), \servv\bigr) = \Dop\bigl( \Dop(\barrv, \Rop(\arrv, \servv)), \Dop(\arrv, \servv)\bigr)$ holds whenever the sequences $\arrv, \barrv, \servv$ are such that the operations are well-defined.  
\end{lemma} 

\begin{proof}
	For the computation choose $(A_j)$ and $(B_j)$ so that $A_j-A_{j-1}=\arr_j$ and $B_j-B_{j-1}=\barr_j$.  Then the output of $\Dop(\barrv, \arrv)$  is  the increment sequence of  
	\[  \wt B_\ell=\sup_{k\le \ell} \Bigl\{  B_k +  \sum_{i=k}^\ell \arr_i\Bigr\} . \]  
	Next,  the output of 
	$\Dop(\Dop(\barrv, \arrv), \servv)$ is  the increment sequence of 
	\[   H_m= \sup_{\ell\le m} \Bigl\{  \wt B_\ell +  \sum_{j=\ell}^m \serv_j\Bigr\}
	= \sup_{k\le m} \Bigl\{  B_k +  \max_{\ell:\, k\le\ell\le m}  \Bigl[ \, \sum_{i=k}^\ell \arr_i +  \sum_{j=\ell}^m \serv_j\Bigr] \Bigr\}. \]
	
	Similarly,  define first 
	\[  \wt A_j=\sup_{k:\,k\le j}  \Bigl\{  A_k+\sum_{i=k}^j \serv_i\Bigr\}
	\quad\text{and}\quad 
	\wc B_\ell=\sup_{k\le \ell} \Bigl\{  B_k +  \sum_{i=k}^\ell \wc\serv_i\Bigr\}.  
	\]
	Then  the output of  $\Dop\bigl( \Dop(\barrv, \Rop(\arrv, \servv)), \Dop(\arrv, \servv)\bigr)$ is  the increment sequence of 
	\[   \wt H_m= \sup_{\ell\le m} \Bigl\{  \wc B_\ell +  \sum_{j=\ell}^m \wt\arr_j\Bigr\}
	= \sup_{k\le m} \Bigl\{  B_k +  \max_{\ell:\, k\le\ell\le m}  \Bigl[ \, \sum_{i=k}^\ell \wc\serv_i +  \sum_{j=\ell}^m \wt\arr_j\Bigr] \Bigr\}. \]
	
	It remains to check that 
	\be\label{DR7} 
	\max_{\ell:\, k\le\ell\le m}  \Bigl[ \, \sum_{i=k}^\ell \wc\serv_i +  \sum_{j=\ell}^m \wt\arr_j\Bigr]
	= \max_{\ell:\, k\le\ell\le m}  \Bigl[ \, \sum_{i=k}^\ell \arr_i +  \sum_{j=\ell}^m \serv_j\Bigr] . 
	\ee
	This can be proved with a case-by-case analysis. See Lemma 4.3 in \cite{fan-sepp-arxiv}. 
\end{proof} 

Specialize to stationary  M/M/1 queues.   Let $\sigma$ be a service rate and $\alpha_1, \alpha_2$ arrival rates.  Assume 
$\sigma>\alpha_1>\alpha_2>0$.  Let $\barrv^1, \barrv^2, \servv$ be   mutually independent i.i.d.\ sequences with marginals   $\barr^k_j\sim\text{Exp}(\alpha_k)$ for $k\in\{1,2\}$   and $\serv_j\sim\text{ Exp}(\sigma)$.    Define a jointly distributed pair of arrival sequences by 
$(\arrv^1, \arrv^2)= \bigl( \barrv^1, \Dop( \barrv^2, \barrv^1)\bigr)$.  From these and   services $\servv$, define jointly distributed output variables: 
\[  
\depav^k=\Dop(\arrv^k,\servv),  \quad \sojov^k=\Sop(\arrv^k,\servv), \quad\text{and} \quad \wc \servv^k=\Rop(\arrv^k,\servv)
\quad \text{  for } k\in\{1,2\}.  
\]  

\begin{lemma}\label{lm:DR5}   We have the following properties. 
	\begin{enumerate}[{\rm(i)}]   \itemsep=3pt 
		\item Marginally $\arrv^2$ is a sequence of i.i.d.\ ${\rm Exp}(\alpha_2)$ variables. 
		\item For  a fixed   $k\in\{1,2\}$ and each  $m\in\Z$, the random variables   $\{\depa^k_j\}_{j\le m}$, $\sojo^k_m$, and $\{\wc\serv^k_j\}_{j\le m}$ are mutually independent with marginal distributions $\depa^k_j\sim\text{\rm Exp}(\alpha_k)$, $\sojo^k_m\sim\text{\rm Exp}(\sigma-\alpha_k)$, and $\wc\serv^k_j\sim\text{\rm Exp}(\sigma)$. 
		\item For  a fixed   $k\in\{1,2\}$, sequences $\depav^k$ and $\wc \servv^k$ are mutually independent sequences of i.i.d.\ random variables with marginal distributions $\depa^k_j\sim\text{\rm Exp}(\alpha_k)$ and $\wc\serv^k_j\sim\text{\rm Exp}(\sigma)$. 
		
		\item  $(\depav^1, \depav^2)\deq(\arrv^1, \arrv^2)$, in other words, we have found a distributional fixed point for this joint queueing operator. 
		
		\item   For any $m\in\Z$, the random variables $\{ \arr^2_i\}_{i\le m}$ and $\{\arr^1_j\}_{j\ge m+1}$ are mutually independent.  
		
		
	\end{enumerate} 
\end{lemma} 

\begin{proof}
	Parts (i)--(iii) are well-known M/M/1 queueing theory.   Proofs can be found for example in Lemma B.2 in Appendix B of \cite{fan-sepp-arxiv}.  
	
	For part  (iv),   the marginal distributions of $\depav^1$ and $\depav^2$ are the correct ones  by Lemma \ref{lm:DR5}(iii). To establish the correct joint distribution,   the definition of $(\arrv^1, \arrv^2)$ points us to find an i.i.d.\ Exp$(\alpha_2)$  random sequence $\zvec$ that is independent of $\depav^1$ and satisfies $\depav^2=   \Dop( \zvec, \depav^1)$.  From the definitions and  Lemma \ref{lm:DR}, 
	\begin{align*}
	\depav^2=\Dop(\arrv^2,\servv) =  \Dop\bigl( \Dop( \barrv^2, \arrv^1),\servv\bigr)
	= \Dop\bigl( \Dop(\barrv^2, \Rop(\arrv^1, \servv)), \Dop(\arrv^1, \servv)\bigr)
	= \Dop\bigl( \Dop(\barrv^2, \wc\servv^1), \depav^1\bigr).
	\end{align*} 
	By assumption $\barrv^2, \arrv^1, \servv$ are independent. Hence by Lemma \ref{lm:DR5}(iii)   $\barrv^2, \wc\servv^1, \depav^1$ are independent.  So we take $\zvec=\Dop(\barrv^2, \wc\servv^1)$  which is an  i.i.d.\ Exp$(\alpha_2)$    sequence by Lemma \ref{lm:DR5}(iii).  This proves part  (iv).
	
	We know that marginally $\arrv^1$ and  $\arrv^2$ are i.i.d.\ sequences.  In queueing language observation (v) becomes obvious. Namely, since $\arrv^2  =    \Dop( \barrv^2, \arrv^1)$, the statement is that past inter-departure   times 
	$\{ \arr^2_i\}_{i\le m}$ are independent of future inter-arrival times $\{\arr^1_j\}_{j\ge m+1}$.  Rigorously,  \eqref{m:800} and \eqref{DSR4} show  that variables $\{ \arr^2_i\}_{i\le m}$ are functions of $(\{\barr^2_i\}_{i\le m}\,, \{\arr^1_i\}_{i\le m})$ which are independent of $\{\arr^1_j\}_{j\ge m+1}$. 
\end{proof}

\section{Coupling and monotonicity in last-passage percolation} \label{app:lpp} 

%

In this section $\w=(\w_x)_{x\in\Z^2}$ is a fixed assignment of real weights.   $G_{x,y}$ is  the last-passage value defined by \eqref{v:G}. No probability is involved.

\begin{lemma}\label{lm:G13}    Suppose weights $\w$ and $\wt\w$ satisfy $\w_{o+ie_1}\ge\wt\w_{o+ie_1}$,  $\w_{o+je_2}\le\wt\w_{o+je_2}$, and $\w_x=\wt\w_x$ for $i,j\ge 1$ and $x\in o+\Z_{>0}^2$.   As in \eqref{v:G} define LPP processes 
	\[  
	\Gpp_{o,y}=\max_{x_{\brbullet}\,\in\,\Pi_{o,y}}\sum_{k=0}^{\abs{y-o}_1}\w_{x_k}
	\quad\text{and}\quad
	\wt\Gpp_{o,y}=\max_{x_{\brbullet}\,\in\,\Pi_{x,y}}\sum_{k=0}^{\abs{y-x}_1}\wt\w_{x_k}
	\quad\text{ for } \; y \in o+\Z_{\geq 0}^2. 
	\] 
	Then for all  
	$y\in o+\Z_{\geq 0}^2$, the increments over nearest-neighbor edges satisfy 
	\[  G_{o,y+e_1}-G_{o,y} \ge \wt G_{o,y+e_1}-\wt G_{o,y}
	\quad\text{and}\quad 
	G_{o,y+e_2}-G_{o,y} \le \wt G_{o,y+e_2}-\wt G_{o,y}.  
	\]
	
\end{lemma} 

\begin{proof}   
	The statements are true by construction  for edges $(y, y+e_i)$ that lie on the  axes $o+\Z_{\ge0}e_i$.   Proceed by induction:  assuming the inequalities hold for the edges $(y,y+e_2)$ and $(y, y+e_1)$ , deduce them for the edges $(y+e_2, y+e_1+e_2)$ and $(y+e_1, y+e_1+e_2)$.  
\end{proof}

\begin{lemma}[Crossing Lemma]\label{lm:crs}
	The inequalities below are valid  whenever the last-passage values are defined. 
	\begin{align}
	G_{o+e_1,\,x+e_2}-G_{o+e_1,\,x} &\leq G_{o,\,x+e_2}-G_{o,\,x}\leq G_{o+e_2,\,x+e_2}-G_{o+e_2,\,x}\label{cr1}\\ 
	G_{o+e_2,\,x+e_1}-G_{o+e_2,\,x} &\leq G_{o,\,x+e_1}-G_{o,\,x}\leq G_{o+e_1,\,x+e_1}-G_{o+e_1,\,x}. \label{cr2}	 	
	\end{align}
\end{lemma}
\begin{proof}
	The proofs of all  parts are similar. We prove the second inequality in \eqref{cr1}, that is, 
	\begin{align}\label{cr in0}
	G_{o,\,x+e_2}-G_{o,\,x}\leq G_{o+e_2,\,x+e_2}-G_{o+e_2,\,x}.
	\end{align}
	The geodesics $\pi_{o,\,x+e_2}$ and  $\pi_{o+e_2,\,x}$ must cross. Let $u$ be the first point where they meet. Note that
	\begin{align}\label{cr in}
	G_{o,u}+G_{u,\,x}\leq G_{o,\,x}\quad \text{and} \quad G_{o+e_2,u}+G_{u,\,x+e_2}\leq G_{o+e_2,\,x+e_2}.
	\end{align}
	Add the two inequalities in \eqref{cr in} and rearrange to obtain \eqref{cr in0}.
	
	This inequality can be proved also from Lemma \ref{lm:G13}, by writing $G_{o+e_2,\,x+e_2}-G_{o+e_2,\,x}=\wt G_{o,\,x+e_2}-\wt G_{o,\,x}$ with environment $\wt\w_{o+y}=\w_{o+y}$ when $y_2>0$ and $\wt\w_{o+ie_1}=-M$ for  large enough  $M$.    
\end{proof}

Fix base points  $u\le v$ on $\Z^2$.   On the quadrant  $v+\Z_{\ge0}^2$,  put a corner weight $\eta_v=0$ and  define boundary weights 
\be\label{eta6} \eta_{v+ke_i} =  G_{u,\,v+ke_i}- G_{u,\,v+(k-1)e_i} 
\qquad\text{for $k\in\Z_{>0}$ and $i\in\{1,2\}$. }\ee
In the bulk use $\eta_x=\w_x$ for $x\in v+\Z_{>0}^2$.    Denote the LPP process in $v+\Z_{\ge0}^2$ that uses weights $\{\eta_x\}_{x\,\in\, v+\Z_{\ge0}^2}$ by 
\be\label{wtG8}   G^{[u]}_{v,\,x}=\max_{x_{\brbullet}\,\in\,\Pi_{v,\,x}}  \sum_{i=0}^{\abs{x-v}_1} \eta_{x_i}, \qquad 
x\in v+\Z_{\ge0}^2. \ee

Assume now that the weights are such that geodesics are unique. Define the  exit point $Z_{u,\,p}$   as in \eqref{exit2}.  For $k\ge 1$ let $Z^{[u]}_{u+ke_1, \,p}$ be the exit point of the geodesic of $G^{[u]}_{u+ke_1, \,p}$. The lemma below follows from taking $v=u+ke_1$ in Lemma \ref{app-lm1}. 

\begin{lemma}\label{lm:shift} 
	For positive integers $m$,   $Z_{u,\,p}=k+m$  if and only if  $Z^{[u]}_{u+ke_1, \,p}=m$.
\end{lemma} 
\section{Convergence of distributions} \label{app:cod}
Let $\mathbb{X}$ be a complete and separable metric space and let $\mathcal{M}_1(\mathbb{X})$ be the set of probability distributions on $\mathbb{X}$. For $\mu,\nu\in\mathcal{M}_1(\mathbb{X})$ and $\epsilon\geq 0$ we define
\begin{align*}
d_{TV_\epsilon}(\mu,\nu)=\inf\{\P\big(|X-Y|>\epsilon\big):\text{$(X,Y)$ is a r.v.\ s.t.\ $X\sim\mu,Y\sim\nu$}\},
\end{align*}
so in particular, for $\epsilon=0$ we obtain the definition of total variation distance of distributions.
\begin{lemma}\label{lem:cod}
	Let $\mu_1,\mu_2,\mu_3\in\mathcal{M}_1(\mathbb{X})$. Suppose that for some $\delta>0$
	\begin{align}\label{d1}
	d_{TV_\epsilon}(\mu_1,\mu_2)\leq \delta\quad \text{and} \quad d_{TV_\epsilon}(\mu_1,\mu_3)\leq \delta.	
	\end{align}
	Then 
	\begin{align*}
	d_{TV_{2\epsilon}}(\mu_2,\mu_3)\leq 2\delta.
	\end{align*}
\end{lemma}
\begin{proof}
	Let $\mathbb{X}_1,\mathbb{X}_2$ and $\mathbb{X}_3$ be three copies of $\mathbb{X}$ and let $\Omega=\mathbb{X}_1\times \mathbb{X}_2 \times \mathbb{X}_3$. By \eqref{d1}, there exist $f_{\mu_1,\mu_2},f_{\mu_1,\mu_3}\in \mathcal{M}_1(\mathbb{X}^2)$ s.t.\ 
	\begin{align*}
	\int_{\mathbb{X}^2}1_{|x-y|>\epsilon}\,df_{\mu_1,\mu_i}(x,y)\leq \delta \quad \text{for $i\in\{2,3\}$}.
	\end{align*}
	For $i\in\{2,3\}$ let $df_{\mu_i|\mu_1}$ be the conditional distribution of $\mu_i$ given  $\mu_1$ w.r.t.\ $f_{\mu_1,\mu_i}$. Define the distribution $F$ on $\Omega$ by
	\begin{align*}
	dF=d\mu_1(x_1)df_{\mu_2|\mu_1}(x_2)df_{\mu_3|\mu_1}(x_3).
	\end{align*}
	Note that the marginals of $F$ are $\mu_1,\mu_2$ and $\mu_3$. Let $\mathcal{P}_{2,3}:\Omega\rightarrow \mathbb{X}^2$ be the projection map of the last two coordinates in $\Omega$ and let $F'$ be the pushforward measure of $F$ with respect to $\mathcal{P}_{2,3}$. Then $F'$ is a coupling of $\mu_2$ and $\mu_3$ and 
	\begin{align}\label{d2}
	&\int_{\mathbb{X}^2}1_{|x-y|>2\epsilon}\,dF'(x,y)=\int_{\Omega}1_{|x_2-x_3|>2\epsilon}\,dF(x_1,x_2,x_3)\\\nonumber
	&=\int_{\Omega}1_{|x_2-x_1+x_1-x_3|>2\epsilon}\,dF(x_1,x_2,x_3)\\\nonumber
	&\leq \int_{\Omega}1_{|x_1-x_3|>\epsilon}\,dF(x_1,x_2,x_3)+\int_{\Omega}1_{|x_2-x_1|>\epsilon}\,dF(x_1,x_2,x_3)\\\nonumber
	&\leq 2\delta.
	\end{align}
	\eqref{d2} implies the result. 
\end{proof}
\begin{lemma}\label{lem:cod2}
	Let $\{\mu_n\}_{n\in\N}$ and $\{\nu_n\}_{n\in\N}$ be two sequences of distributions in a metric space  $\mathbb{X}$. Suppose we have the following weak convergences of distributions
	\begin{align}\label{d3}
	&\mu_n\rightarrow \mu\\
	&\nu_n\rightarrow \nu.\nonumber
	\end{align}
	Also assume that for every $n\in \N$
	\begin{align}\label{d4}
	d_{TV}(\mu_n,\nu_n)\leq \delta.
	\end{align}
	Then 
	\begin{align*}
	d_{TV}(\mu,\nu)\leq 3\delta.
	\end{align*}
\end{lemma}
\begin{proof}
	Fix $k\in\N$. The convergences in \eqref{d3} can be realized a.s.\ and so, there exists $N(k)$ such that for $n>N(k)$
	\begin{align*}
	&d_{TV_{k^{-1}}}(\mu_n,\mu)\leq \delta\\
	&d_{TV_{k^{-1}}}(\nu_n,\nu)\leq \delta.
	\end{align*}
	Using Lemma \ref{lem:cod} twice with \eqref{d4} implies that for $n>N(k)$
	\begin{align*}
	d_{TV_{k^{-1}}}(\mu,\nu)\leq 3\delta,
	\end{align*} 
	and so there must be a coupling $F^k$ of $\mu$ and $\nu$ such that
	\begin{align*}
	\int 1_{|x-y|>3k^{-1}} \,dF^k \leq 3 \delta.
	\end{align*}
	The sequence $\{F^{k}\}_{k\in\N}$ is tight with respect to the product metric as the marginals of $F^{k}$ are independent of $k$. It follows that there must be a weakly convergent subsequence $F^{k_m}$ such that
	\begin{align*}
	F^{k_m}\rightarrow F,
	\end{align*}
	where $F$ is a coupling of $\mu$ and $\nu$ and that for every $k\in\N$
	\begin{align*}
	\int 1_{|x-y|>3k^{-1}} \,dF \leq 3 \delta.
	\end{align*}
	Sending $k$ to infinity implies the result.
\end{proof}
\newpage
\bibliographystyle{plain}

\bibliography{../Timo_old_bib}
\end{document}